\documentclass[a4paper,11pt]{article}
\usepackage{blindtext}

\usepackage{subfiles} 
\usepackage[pagewise]{lineno}
\usepackage{cite,titlesec,graphics,amsmath, amssymb, amscd, amsthm, euscript, hyperref,
indentfirst,fancyhdr,pgfplots,stmaryrd, textcomp, enumerate, mathrsfs, esint, mathtools} 
\usetikzlibrary{arrows.meta,calc,bending,decorations.markings,decorations.pathmorphing}

\graphicspath{{.}{./eps/}}

\hypersetup{
    colorlinks=true,
    linkcolor=blue,
    filecolor=magenta,      
    urlcolor=cyan,
    pdftitle={Stretch on the affine-additive},
    pdfpagemode=FullScreen,
    }
\makeatletter \theoremstyle{plain}

\newcommand{\R}{\mathbb R}
\newcommand{\N}{\mathbb N}

\newcommand{\Aa}{\mathcal{AA}}

\newcommand{\C}{\mathbb C}

\newcommand{\bH}{\mathbf H}

\newcommand{\Mod}{\mathrm{Mod}}

\newtheorem{thm}{Theorem}[section]
\newtheorem{cor}[thm]{Corollary}
\newtheorem{prop}[thm]{Proposition}
\newtheorem{lem}[thm]{Lemma}
\theoremstyle{definition}
\newtheorem{defn}[thm]{Definition}

\newtheorem{rem}[thm]{Remark}
\DeclareMathOperator{\esssup}{ess \sup}

\DeclareMathOperator{\Adm}{\rm Adm}

\makeatother

\newcommand\blfootnote[1]{%
  \begingroup
  \renewcommand\thefootnote{}\footnote{#1}%
  \addtocounter{footnote}{-1}%
  \endgroup
}

\providecommand{\keywords}[1]
{
  \small	
  {\textit{Key words and phrases.}} #1
}

\providecommand{\subjclass}[1]
{
  \small	
  {\textit{$2020$ Mathematics Subject Classifications.}} #1
}

\begin{document}

\title{Stretch maps on the affine-additive group}

\author{Zolt\'an M. Balogh
  \and
  Elia  Bubani
  \and
  Ioannis D. Platis
}

\newcommand{\Addresses}{{
  \bigskip
  \footnotesize

  Z.M.~Balogh, (Corresponding author), \textsc{Universit\"at Bern,
Mathematisches Institut (MAI), \\
Sidlerstrasse 5,
3012 Bern,
Switzerland.}\par\nopagebreak
  \textit{E-mail address}, Z.M.~Balogh: \texttt{zoltan.balogh@unibe.ch}

  \medskip

  E.~Bubani,  \textsc{Universit\"at Bern,
Mathematisches Institut (MAI), 
Sidlerstrasse 5,
3012 Bern,
Switzerland.}\par\nopagebreak
  \textit{E-mail address}, E.~Bubani: \texttt{elia.bubani@unibe.ch}

  \medskip

  I.D.~Platis, \textsc{University of Patras, Department of Mathematics, Rion, 26504 Patras, Greece}\par\nopagebreak
  \textit{E-mail address}, I.D.~Platis: \texttt{idplatis@upatras.gr}

}}
\date{}

\maketitle
\centerline{\textit{Dedicated to the memory of Alexander Vasil'ev}}

\begin{abstract}
We define linear and radial stretch maps in the affine-additive group, and prove that they are minimizers of the mean quasiconformal distortion functional. For the proofs we use a method based on the notion of modulus of a curve family and the minimal stretching property (MSP) of the afore-mentioned maps. MSP relies on certain given curve families compatible with the respective geometric settings of the strech maps.
\end{abstract}

\blfootnote{\date{\today}}
\blfootnote{\subjclass{53C17, 30L10}}
\blfootnote{\keywords{Quasiconformal maps, Sub-Riemannian metric, Heisenberg group}}
\blfootnote{\thanks{This research was supported by the Swiss National Science Foundation, Grants Nr. 191978 and 228012}}

\setcounter{tocdepth}{1}
\tableofcontents 

\newpage

\section{Introduction and statement of the main results}
The Gr\"otzsch problem can be formulated as follows: let $a>1$,
consider a square $Q=(0,1)\times(0,1)$ and a rectangle $R=(0,a)\times(0,1)$. 
We ask if there is a conformal map which maps the horizontal edges of $Q$ into the
corresponding horizontal edges of $R_a$, and requiring respectively the same condition on the vertical edges. It turns out that there is no such conformal mapping; however, using complex notation one finds that the linear stretch map given by
$$
x+iy\mapsto ax+iy,
$$ 
solves the Gr\"otzsch problem and it is the closest to be conformal. The works of Gr\"otzsch \cite{G}, \cite{AP}, established criteria to measure how to approximate conformality. Gr\"otzsch also formulated the analogous problem between annuli in the complex plane and proved that the solution is the radial stretch map which, for $k>0$, is given by 
$$
z\mapsto|z|^{k-1}z.
$$ 
Via the formal substitution $z=e^{\xi+i\psi}$ one sees that the radial stretch can be seen as the linear map given by $(\xi,\psi)\mapsto(k\xi,\psi)$. It is also worth mentioning that Astala \cite{Ast}, used the radial stretch map to prove the sharpness on the optimal Sobolev exponent for $K$-quasiconformal mappings in the complex plane. \\
An \textit{extremal quasiconformal} map is a minimizer for the maximal distortion among some class of quasiconformal mappings. 
Methods involving the modulus of curve families were used to identify such extremality between annuli in the complex plane by Balogh, F\"assler and Platis \cite{BFP-C}. 
Motivated by Mostow's rigidity \cite{M-R}, the Heisenberg groups became an interesting setting for studying quasiconformal maps. The related theory reached an advanced study thanks to Pansu \cite{Pansu}, and Kor\'anyi-Reimann \cite{KR1}, \cite{KR2}. We underline that as in the Euclidean case (see Gehring and V\"ais\"al\"a \cite{GV}), quasiconformal maps of the Heisenberg group can distort the Hausdorff dimension with arbitrary fashion, see Balogh \cite{B}.
Balogh, F\"assler and Platis constructed appropriate analogues of linear and radial stretch maps for the first Heisenberg group, \cite{BFP-H-MM}. For the latter case the same authors subsequently proved that the radial stretch map is an essentially unique minimizer for the \textit{mean distortion functional}, \cite{BFP-H-U}.\\
In this paper we search for extremal quasiconformal maps on a different metric measure space: that is, the affine-additive group, which from the topological viewpoint is one of the eight $3$-dimensional Thurston geometries, see \cite{Thu}. 
In \cite{BBP1} the authors prove that the affine-additive group is conformally hyperbolic. This is in contrast with the Heisenberg group which is conformally parabolic, see Zorich in \cite{Z}. In particular, the two spaces are not quasiconformally equivalent. The main obstacle given by this inequivalence is that the maps which are quasiconformal in the first Heisenberg group are not compatible with the quasiconfomal maps of the affine-additive group. Therefore becomes interesting to study the quasiconformal maps on the affine-additive group as well as to seek extremal maps for a mean distortion functional adapted to the metric and the measure structures of the considered space. We implement a suitable version of the modulus method which is evidently useful for our purposes. As a consequence we are able to present new examples of self-mappings of the affine-additive group which share some remarkable features. 

Following the work of Balogh, F\"assler and Platis \cite{BFP-H-MM}, we define a mapping having the "minimal stretching property" (MSP) for a given curve family by adapting it to our particular case in the sub-Riemannian framework.

In the first part we build a modulus method which relies on the minimal stretching property of the map for a given curve family foliating the domain of definition.

In the second and third parts we show that this designed method detects quasiconformal maps between domains of the affine-additive group. The effect is that we obtain extremal stretch maps minimizing the mean distortion functional in the class of all quasiconformal mappings between two such domains.\\

We begin by describing the affine-additive group; in particular, its sub-Riemannian structure, and a preliminary metric notion of quasiconformal mappings for this group. For more details about the affine-additive group, we refer to \cite{Bubani}.

Our starting point is the hyperbolic plane, defined as 
\begin{equation*}
    \bH_\C^1:=\{x+iy\in\C:x>0,\,y\in\R\}.
\end{equation*}
For fixed $\lambda>0,\,t\in\R$, we consider affine transformations $f_{\lambda,t}:\bH_\C^1\to\bH_\C^1$, defined by
$$
f_{\lambda,t}(x+iy)=\lambda(x+iy)+it.
$$
We observe that $\bH_\C^1$ is in bijection with the set of transformations of the above form: to each point $\lambda+it$ we uniquely assign the transformation $f_{\lambda,t}$. Therefore we define a group structure on $\bH_\C^1$ by considering the composition of any two transformations $f_{\lambda',t'}$ and $f_{\lambda,t}$:
\begin{eqnarray*}
(f_{\lambda',t'}\circ f_{\lambda,t})(x+iy) = f_{\lambda'\lambda,\;\lambda't+t'}(x+iy).
\end{eqnarray*}
Based on this we introduce the group operation on $\bH_\C^1$ by
\begin{equation*}
    (\lambda'+it')\star_0(\lambda+it)=\lambda'(\lambda+it)+it'.
\end{equation*}

This operation is extended over the space 
$
\R\times\bH_\C^1
$ as follows: we take the Cartesian product of the additive group $(\R,+)$ and the group $(\bH_\C^1,\star_0)$. Then, if $p'=(a',\lambda'+it')$ and  $p=(a,\lambda+it)$ belong to $\R\times\bH_\C^1$, we have
\begin{equation}\label{eq-mult}
p'\star p=(a'+a,\lambda'(\lambda+it)+it')\in\R\times\bH_\C^1.
\end{equation}
The pair $\Aa=(\R\times\bH_\C^1,\;\star)$ shall be called the \textit{affine-additive group}.\\
We define the following $1$-form on $\Aa$: 
\begin{equation}\label{contact form omega}
\vartheta=\frac{dt}{2\lambda}-da.    
\end{equation}
Then the pair $(\Aa,\vartheta)$ is a contact manifold.  A basis for the Lie algebra comprises the following left invariant vector fields:
 \begin{equation*} U =\partial_a+2\lambda\partial_t, \  V =2\lambda\partial_\lambda, \ W =-\partial_a.
\end{equation*}
The horizontal tangent bundle of $\Aa$ is 
$$\mathcal{H}_{\Aa}=\ker\vartheta=\text{Span}\{U,V\}.$$
We denote by $\langle\cdot,\cdot\rangle_{\Aa}$ the sub-Riemannian metric on $\mathcal{H}_{\Aa}$ which makes $\{ U, V\}$ an orthonormal basis. An absolutely continuous curve $\gamma:[c,d]\rightarrow\Aa$ is called \textit{horizontal} if and only if $\dot\gamma(s)\in\ker\vartheta_{\gamma(s)}$ for almost every $s\in[c,d]$ and its horizontal velocity is given by $|\dot\gamma|_H=\sqrt{\langle\dot\gamma,\dot\gamma\rangle_\Aa}$. The associated \textit{sub-Riemannian distance} is denoted by $d_\Aa$.
The left-invariant Haar measure for $\Aa$ is
$
d\mu_\Aa=\frac{1}{\lambda^2}\,da\, d\lambda\, dt.
$
The affine-additive group as a metric measure space shall be denoted by $(\Aa, d_{\Aa}, \mu_{\Aa})$. The authors proved in \cite{BBP1} that the conformal type of the affine-additive group is hyperbolic.\\ 
A \textit{metric definition} of quasiconformality (in the sense of Heinonen and Koskela \cite{HK}) is given in terms of the sub-Riemannian distance on the affine-additive group as follows. If $\Omega,\Omega'$ are two domains in $\Aa$, a homeomorphism $f: \Omega\to \Omega'$ is called \textit{quasiconformal} if there exists $1\le H<\infty$ such that 
\begin{equation}\label{metric qc}
    \limsup_{r\rightarrow0}\frac{\sup_{d_\Aa(p,q)\leq r}d_\Aa(f(p),f(q))}{\inf_{d_\Aa(p,q)\geq r}d_\Aa(f(p),f(q))}=:H_f(p)\leq H,\quad \text{for all }p \in \Aa.
\end{equation}
There also exist an analytic as well as a geometric definition for quasiconformal mappings of $\Aa$. It turns out that both are equivalent to the above metric definition. This fact was already well-known for quasiconformal maps of $\C$ and also known for general sub-Riemannian manifolds; we provide an Appendix which illustrates the details for the case of $\Aa$. Such quasiconformal mappings share Sobolev regularity properties and also satisfy the contact condition, meaning that they preserve the contact form $\vartheta$, i.e. $f^*\vartheta=\sigma\vartheta$ almost everywhere for some non-vanishing smooth function $\sigma$. This imposes some rigidity on the quasiconformal mappings of smooth type; on the other hand, the contact condition is a quite straightforward requirement for a candidate map. To be more precise, let $f=(f_1,f_2+if_3):\Omega\to\Omega'$ be a quasiconformal mapping between domains in the affine-additive group, and let $f_I=f_2+if_3$. By defining the complex vector fields 
\begin{equation}\label{CVF}
Z=\frac{1}{2}(V-iU) ,\ \overline Z=\frac{1}{2}(V+iU),    
\end{equation}
it turns out that the horizontal derivatives 
$Zf_I\text{ and }\overline Zf_I$
exist both as distributions and almost everywhere. From now on, we will consider quasiconformal mappings to be \textit{orientation preserving}, i.e., $|Zf_I(p)|>|\overline{Z}f_I(p)|\text{ for almost every }p\in\Omega$. \\
We then define the \textit{Beltrami coefficient} and the \textit{distortion quotient} as 
$$
\mu_f(p)=\frac{\overline{Z}f_I}{Zf_I}(p) 
\text{ and }
 K(p,f)=\frac{|Zf_I|+|\overline{Z}f_I|}{|Zf_I|-|\overline{Z}f_I|}(p),
$$ for points $p\in\Aa$ where these expressions exist. 
In this paper we shall make an extended use of the square of the distortion quotient $K(p,f)^2$.   
By letting $K_f=\esssup_p K(p,f)$, we underline that any smooth contact transformation $f$ with $1\leq K_f<\infty$ is quasiconformal.\\
Given two domains $\Omega,\Omega'\subseteq\Aa$ and a certain given class $\mathcal{F}$ comprising quasiconformal mappings $f:\Omega\to\Omega'$, we may define the deviation of a quasiconformal map from conformality as follows. We say that an $f_0\in\mathcal{F}$ is \textit{extremal for a mean distortion functional} if  
\begin{equation}\label{MDF}
\int_\Omega K^2(p,f_0)\rho_0^4(p)\,d\mu_{\Aa}(p)=\min_{f\in\mathcal{F}}\int_\Omega K^2(p,f)\rho_0^4(p)\,d\mu_{\Aa}(p),    
\end{equation}
for a given density $\rho_0$. This $\rho_0$ is extremal for the modulus of a chosen curve family foliating the domain $\Omega$.\\
The modulus $\Mod_4(\Gamma)$ of a curve family $\Gamma$ is defined as follows. Let $\Adm(\Gamma)$ be the set of \textit{admissible densities}: that is, non-negative Borel
functions $\rho:\Aa\to[0,\infty]$ such that $\int_\gamma\rho\,d\ell\ge1$ for all rectifiable curves $\gamma\in\Gamma$.
Rectifiability here is understood in terms of the sub-Riemannian distance $d_\Aa$.
Then 
\begin{equation}\label{def 4-Mod}
    \Mod_4(\Gamma)=\inf_{\rho\in\Adm(\Gamma)}\int_\Aa \rho(p)^4\,d\mu_\Aa(p),
\end{equation}
see the Appendix \ref{Appendix} for details.\\
It is important to mention that the $\rho_0$ used in \eqref{MDF} corresponds to the extremal density which attains the infimum for $\Mod_4(\Gamma)$.
For example, we mention that the modulus of the curve family connecting the two boundaries of any revolution ring in the first Heisenberg group has been computed by Platis in \cite{P-RR}.
The modulus method and its applications to
extremal problems for conformal, quasiconformal mappings and the extension of
moduli onto Teichm\"uller spaces is treated in the book of Vasil'ev \cite{Vas}.
Moreover, the notion of modulus of a curve family has been extended to serve as further quasiconformal invariants in the works of Brakalova, Markina and Vasil'ev in \cite{BMV-mod} and in \cite{BMV-extr}. \\
An orientation preserving quasiconformal map  $f_0:\Omega\rightarrow\Omega'$ between domains in $\Aa$ has the \textit{minimal stretching property (MSP)} for a family $\Gamma_0$ of horizontal curves in $\Omega$ if for all $\gamma\in\Gamma_0$, $\gamma:[c,d]\to\Aa$, one has
    \begin{equation*}
    \mu_{f_0}(\gamma(s))\frac{\dot{\overline{\gamma_I}}(s)}{\dot{\gamma_I}(s)}<0 \text{ for almost every } s\in(c,d) \text{ with }\mu_{f_0}(\gamma(s))\neq0.    
    \end{equation*}
Note that in the latter definition we require the expression $\mu_{f_0}(\gamma(s))\frac{\dot{\overline{\gamma_I}}(s)}{\dot{\gamma_I}(s)}$ to be real-valued.\\ 
Suppose next that $\Delta$ is a domain in $\R^2$. Let $0\le c<d$ and let $\gamma:(c,d)\times\Delta \rightarrow\Omega$ be a diffeomorphism which foliates a bounded domain $\Omega$ in the affine-additive group with the property that $$\gamma(\cdot,\delta):[c,d]\rightarrow\overline\Omega$$ is a horizontal curve with $|\dot\gamma(s,\delta)|_H\neq0$ for all $\delta\in \Delta$ and $$d\mu_{\Aa}(\gamma(s,\delta))=|\dot\gamma(s,\delta)|_H^4\,ds\,d\nu(\delta)$$ for a measure $\nu$ on $\Delta$. We define the curve family $\Gamma_0=\{\gamma(\cdot,\delta):\delta\in \Delta\}$ and it will be shown that 
\begin{equation}\label{rho_0}
    \rho_0(p)=\left\{\begin{matrix}
    \frac{1}{(d-c)|\dot\gamma(\gamma^{-1}(p))|_H},& &p=\gamma(s,\delta)\in\Omega,\\
    \\
    0,& &p\notin\Omega,
  \end{matrix}\right.
\end{equation}
is an extremal density for  $\Mod_4(\Gamma_0)$.\\
Let $f_0:\Omega\to\Omega'$ be an orientation preserving quasiconformal mapping between domains in the affine-additive group. Let $\gamma$ be a foliation of $\Omega$ as described above. Assume as well that $f_0$ has the MSP for $\Gamma_0$; we then say that the distortion quotient $K(\cdot,f_0)$ is \textit{constant along every curve} $\gamma$ if and only if
\begin{equation}\label{K const on curves}
K(\gamma(s,\delta),f_0)\equiv K_{f_0}(\delta)\quad \text{for all } (s,\delta)\in(c,d)\times \Delta.    
\end{equation}

The following condition for extremality of the mean distortion integral is the main result from the first part of this paper.
\begin{thm}\label{Thm1}
    Assume that $f_0$ satisfies the minimal stretching property with respect to $\Gamma_0$ described as above. Let $\rho_0$ be the extremal density for $\Gamma_0$ and assume $K(\cdot,f_0)$ to be constant along every curve foliating $\Omega$. Let $\Gamma\supseteq\Gamma_0$ be a curve family such that $\rho_0\in\Adm(\Gamma)$ and let $\mathcal{F}$ be the class of quasiconformal maps $f:\Omega\rightarrow\Omega'$ such that 
    \begin{equation*}
    \Mod_4(f_0(\Gamma_0))\leq\Mod_4(f(\Gamma)).    
    \end{equation*}
Then 
\begin{equation*}
    \int_\Omega K^2(p,f_0)\rho_0^4(p)\,d\mu_{\Aa}(p)\leq\int_\Omega K^2(p,f)\rho_0^4(p)\,d\mu_{\Aa}(p)
\end{equation*}
for all $f\in\mathcal{F}$.
\end{thm}
Towards a first application of Theorem \ref{Thm1}, we define certain suitable domains in the affine-additive group.
Let $k>0$ and consider two domains $\Omega$ and $\Omega^k$ which shall be defined in detail in Section \ref{linear stretch}.
Further, consider the curve family  $\Gamma_0$ foliating $\Omega$ as well as its extremal density $\rho_0$ given by \eqref{rho_0}.\\
The extremal mapping $f_0$ will be the \textit{linear stretch map} $f_k:\overline{\Omega}\to \overline{\Omega^k}$: 
$$
f_k(a,\lambda+it)=(ka,\lambda+ikt).$$
We next formulate the main result of the second part. Denote by $\mathcal{F}_k$ the
class of all quasiconformal maps $\overline{\Omega}\to \overline{\Omega^k}$ with prescribed boundary conditions which will be rigorously set up in Section \ref{linear stretch}.
\begin{thm}\label{Thm1.5}
The linear stretch map $f_k:\Omega\to\Omega^k$ is an orientation preserving quasiconformal map. With the above notation for $\rho_0$, $f_k$ minimizes the mean distortion within the class $\mathcal{F}_k$: for all $f\in\mathcal{F}_k$ we have that
\begin{equation}
    K^2_{f_k}\leq\frac{\int_\Omega K^2(\cdot,f)\rho_0^4\,d\mu_{\Aa}}{\int_\Omega \rho_0^4\,d\mu_{\Aa}}.
\end{equation}
\end{thm}
Towards another application of Theorem \ref{Thm1}, we can also define some suitable domain in the affine-additive group, which looks natural, once we have considered a different type of coordinate system on the affine-additive group.
The \textit{cylindrical-logarithmic coordinates} are defined as 
\begin{equation*}
    (a,\lambda+it)=\left(a,e^{\xi+i\psi}\right),\quad(a,\xi,\psi)\in\R\times\R\times\left(-\frac{\pi}{2},\frac{\pi}{2}\right).
\end{equation*}
Let $0<k<1$, $r_0>1$ and $0<\psi_0<\frac{\pi}{2}$. Consider two truncated cylindric shells: $D_{r_0,\psi_0}$ and $D^k_{r_0,\psi_0}$, see for details Section \ref{trunc.td cyl shell}.
Further, consider the curve family $\Gamma_0$ foliating $D_{r_0,\psi_0}$ as well as its extremal density $\rho_0$ given by \eqref{rho_0}.\\
The extremal mapping $f_0$ will be the \textit{radial stretch map} $f_k:\overline{D_{r_0,\,\psi_0}}\to \overline{D^k_{r_0,\,\psi_0}}$, which in cylindrical-logarithmic coordinates reads as
$$
(a,\xi,\psi)\mapsto\left(a-\frac{\psi}{2}+\frac{1}{2}\tan^{-1}\left(\frac{\tan\psi}{k}\right),k\xi,\tan^{-1}\left(\frac{\tan\psi}{k}\right)\right).
$$
Finally, we formulate the main result of the third part of our work. Denote by $\mathcal{F}_k$ the
class of all quasiconformal maps $\overline{D_{r_0,\,\psi_0}}\to \overline{D^k_{r_0,\,\psi_0}}$ with prescribed boundary conditions (see Section \ref{trunc.td cyl shell} for details).
\begin{thm}\label{Thm2}
The radial stretch map $f_k:\overline{D_{r_0,\,\psi_0}}\to \overline{D^k_{r_0,\,\psi_0}}$ is an orientation preserving quasiconformal map. With the above notation for $\rho_0$, $f_k$ minimizes the mean distortion within the class $\mathcal{F}_k$: for all $f\in\mathcal{F}_k$ we have that
$$
\int_{D_{r_0,\psi_0}}K(p,f_k)^2\rho_0(p)^4\,d\mu_\Aa(p)\leq\int_{D_{r_0,\psi_0}}K(p,f)^2\rho_0(p)^4\,d\mu_\Aa(p)\,.
$$
\end{thm}
\noindent{\textbf{Structure of the paper.}} In Section \ref{Sec2} we define the minimal stretching property for a quasiconformal mapping with respect to some curve family, then we formulate a criterium establishing if such mapping is a minimizer for the mean distortion functional. In Section \ref{linear stretch} we present two geometric settings where the linear stretch map is minimizing the maximal distortion. In Section \ref{Sec4} we construct cylindrical-logarithmic coordinates for the affine-additive group and examine their properties. In Section \ref{trunc.td cyl shell} we show that the radial stretch map is a minimizer for the mean distortion functional. In Section \ref{OpPb} we formulate an open question arising as a consequence of the last theorem. We additionally provide an Appendix \ref{Appendix} in which we introduce and explain all the background material.


\section{Minimal stretching property and extremality}\label{Sec2}
We begin this section by describing the sub-Riemannian structure of $\Aa$ in detail.
Recall that an absolutely continuous curve $\gamma:[c,d]\to\Aa$, $\gamma(s)=(a(s),\lambda(s)+it(s)),\\s\in[c,d]$, is horizontal when $\dot\gamma(s)\in\ker\vartheta_{\gamma(s)}$ for almost every $s\in[c,d]$. This is equivalent to the o.d.e.
\begin{equation}\label{horiz ode}
    \frac{\dot t(s)}{2\lambda(s)}-\dot a(s)=0, \text{ a.e. } s\in[c,d].
\end{equation}
The horizontal velocity $|\dot\gamma|_H$ of $\gamma$ is defined by the relation
\begin{equation}\label{horiz speed}
|\dot\gamma|_H=\left(\langle\dot\gamma,U\rangle_{\Aa}^2+\langle\dot\gamma,V\rangle_{\Aa}^2\right)^{1/2}=\frac{\sqrt{\dot \lambda^2+\dot t^2}}{2\lambda}.
\end{equation}
Let $\pi:\Aa\to\bH^1_\C$ be the canonical projection given by $\pi(a,\lambda+it)=\lambda+it,\,(a,\lambda+it)\in\Aa$. The horizontal length of $\gamma$ is then given by
\begin{equation}
 \ell(\gamma)=\int_c^d \frac{\sqrt{\dot \lambda^2+\dot t^2}}{2\lambda}ds=\int_c^d \frac{|\dot{\gamma_I}|}{2\lambda}ds,
\end{equation}
where $\gamma_I=\pi\circ\gamma$.\\
The corresponding Carnot-Carath\'eodory distance $d_{\Aa}$ associated to the sub-Riemannian metric $\langle\cdot,\cdot\rangle_\Aa$ is defined for all $p, q  \in\Aa$ as follows:
\begin{equation*}
    d_{\Aa}(p,q)=\inf_{\gamma\in \Gamma_{p,q}}\{\ell(\gamma) \},
\end{equation*}
where $\Gamma_{p,q}$ is the following family of curves: 
$$
\Gamma_{p,q}=\{\gamma,\gamma:[0,1]\to\Aa \text{ horizontal and such that}\; 
       \gamma(0)=p,\;
       \gamma(1)=q\}.
$$
We underline here that since $\mathcal{H}_\Aa$ is bracket generating, the distance $d_{\Aa}$ is finite, geodesic and
induces the manifold topology (see Mitchell \cite{M}, Montgomery \cite{Mont}).
We are now able to give a result concerning an extremal density for the modulus of a curve family foliating a bounded domain in the affine-additive group.
\begin{prop}\label{P Mod_4 Gamma_0}
    Suppose $\Delta$ is a domain in $\R^2$. Let $0\le c<d$ and let $$\gamma:(c,d)\times\Delta \rightarrow\Omega$$ be a diffeomorphism which foliates a bounded domain $\Omega$ in the affine-additive group with the property that $$\gamma(\cdot,\delta):[c,d]\rightarrow\overline\Omega$$ is an horizontal curve with $|\dot\gamma(s,\delta)|_H\neq0$ for all $\delta\in \Delta$ and 
    $$
    d\mu_{\Aa}(\gamma(s,\delta))=|\dot\gamma(s,\delta)|_H^4\,ds\,d\nu(\delta)
    $$ for a measure $\nu$ on $\Delta$. Then 
\begin{equation}\label{rho_0 strategy}
    \rho_0(p)=\left\{\begin{matrix}
    \frac{1}{(d-c)|\dot\gamma(\gamma^{-1}(p))|_H},& &p=\gamma(s,\delta)\in\Omega,\\
    \\
    0,& &p\notin\Omega,
  \end{matrix}\right.
\end{equation}
is an extremal density for the curve family $\Gamma_0=\{\gamma(\cdot,\delta):\delta\in \Delta\}$ with $$\Mod_4(\Gamma_0)=\frac{1}{(d-c)^{3}}\int_\Delta d\nu(\delta).$$ 
Here, $\dot\gamma(s,\delta)=\frac{\partial}{\partial s}\gamma(s,\delta)$ for $(s,\delta)\in(c,d)\times\Delta$.
\end{prop}
\begin{proof}
    We show first that $\rho_0\in \text{Adm}(\Gamma_0)$: this is because for any $\gamma(\cdot,\delta)\in\Gamma_0$ we have
\begin{equation*}
    \int_{\gamma(\cdot,\delta)}\rho_0\, d\ell=\int_c^d \rho_0(\gamma(s,\delta))|\dot\gamma(s,\delta)|_H\,ds=1.
\end{equation*}
Since we assume the measure decomposition $d\mu_{\Aa}(\gamma(s,\delta))=|\dot\gamma(s,\delta)|_H^4\,ds\,d\nu(\delta)$, a direct computation yields 
\begin{align*}
    \int_\Omega \rho_0^4(p)\,d\mu_{\Aa}(p)=&\int_\Delta\int_c^d\rho_0^4(\gamma(s,\delta))|\dot\gamma(s,\delta)|_H^4\,ds\,d\nu(\delta)\\
    =&\frac{1}{(d-c)^{3}}\int_\Delta d\nu(\delta)\,. 
\end{align*}
Consequently, $$\Mod_4(\Gamma_0)\leq\frac{1}{(d-c)^{3}}\int_\Delta d\nu(\delta).$$ 
For the reverse inequality, consider an arbitrary density $\rho\in\text{Adm}(\Gamma_0)$. By using the admissibility of $\rho$ and then H\"older's inequality with conjugated exponents $4$ and $\frac{4}{3}$, we have
\begin{align*}
    1&\leq\int_{\gamma(\cdot,\delta)}\rho \,d\ell=\int_c^d\rho(\gamma(s,\delta))|\dot\gamma(s,\delta)|_H\,ds\\
    &\leq\left(\int_c^d\rho^4(\gamma(s,\delta))|\dot\gamma(s,\delta)|_H^4\,ds\right)^{\frac{1}{4}}\left(d-c\right)^{\frac{3}{4}}\,,
\end{align*}
for every $\delta\in \Delta$. Thus
\begin{align*}
   \frac{1}{(d-c)^{\frac{3}{4}}}\leq\left(\int_c^d\rho^4(\gamma(s,\delta))|\dot\gamma(s,\delta)|_H^4\,ds\right)^{\frac{1}{4}}.
\end{align*}
We raise the latter inequality to the $4$-th power and then we integrate with respect to $d\nu$ over $\Delta$ to eventually obtain
\begin{align*}
    \frac{1}{(d-c)^3}\int_\Delta d\nu(\delta)\leq\int_\Delta\int_c^d\rho^4(\gamma(s,\delta))|\dot\gamma(s,\delta)|_H^4\,ds\,d\mu(\delta)=\int_\Omega \rho^4(p)\,d\mu_{\Aa}(p).
\end{align*}
Our result follows by taking the infimum over all densities $\rho\in\text{Adm}(\Gamma_0)$.
\end{proof} 


\subsection{Minimization of the mean distortion.}
Following the work of Balogh, F\"assler and Platis \cite{BFP-H-MM}, we define the minimal stretching property as follows:
\begin{defn}\label{MSP-def}
    We say that an orientation preserving quasiconformal map  $f_0:\Omega\rightarrow\Omega'$ between domains in $\Aa$ has the \textit{minimal stretching property (MSP)} for a family $\Gamma_0$ of horizontal curves in $\Omega$ if for all $\gamma\in\Gamma_0$, $\gamma:[c,d]\to\Aa$, one has
    \begin{equation}\label{MSP-eq}
    \mu_{f_0}(\gamma(s))\frac{\dot{\overline{\gamma_I}}(s)}{\dot{\gamma_I}(s)}<0 \text{ for almost every } s\in[c,d] \text{ with }\mu_{f_0}(\gamma(s))\neq0.    
    \end{equation}
\end{defn}
If a map $f_0$ has the MSP for a curve family $\Gamma_0$, this means geometrically that $\Gamma_0$ consists of curves which are tangential to the direction of the least stretching of $f_0$. To make this precise, we state and prove the following 
\begin{lem}\label{chain rule}
    Let $f=(f_1,f_2+if_3):\Omega\to\Aa$ be a quasiconformal map on a domain $\Omega\subseteq\Aa$. Let $\Gamma$ be a curve family:
    $$
    \Gamma=\{\gamma(\cdot)=(a(\cdot),\lambda(\cdot)+it(\cdot)),\gamma:[c,d]\to\Omega\text{ horizontal}\,\}.
    $$ Then there exists a sub-family $\Gamma'\subset\Gamma$ of curves with $\Mod_4(\Gamma')=0$ and such that
    \begin{equation}\label{f tang curve}
        (f_I\circ\gamma)^\cdot(s)=\frac{1}{2\lambda(s)}\left(Zf_I(\gamma(s))\dot{\gamma_I}(s)+\overline{Z}f_I(\gamma(s))\dot{\overline{\gamma_I}}(s)\right) \quad  \text{for a.e.}\quad s\in(c,d),
    \end{equation}
     for all $\gamma\in\Gamma\setminus\Gamma'$.
\end{lem}
\begin{proof}
If $\gamma\in\Gamma$ is absolutely continuous then since $f$ is quasiconformal, we have that the image $f\circ\gamma$ is absolutely continuous up to a sub-family of curves having zero $4$-modulus, see \cite{BKR}. Therefore we can choose $\gamma$ such that $f\circ\gamma$ is differentiable almost everywhere. Choosing such an absolutely continuous curve $\gamma:[c,d]\to\Aa$ given by $\gamma(s)=(a(s),\lambda(s)+it(s))$, by using the chain rule we can write
\begin{align*}
    (f_I\circ\gamma)^\cdot(s)=&\nabla f_2(\gamma(s))\cdot\dot{\gamma}(s)+i\,\nabla f_3(\gamma(s))\cdot\dot{\gamma}(s),\quad\text{for a.e. } s\in(c,d).
\end{align*}
Using the o.d.e. \eqref{horiz ode} which holds for $\gamma$ as well as the identities
$$
\dot \lambda(s)=\frac{\dot\gamma_I(s)+\dot{\overline{\gamma_I}}(s)}{2},\quad \dot t(s)=\frac{\dot\gamma_I(s)-\dot{\overline{\gamma_I}}(s)}{2i},
$$
we obtain
$$
(f_I\circ\gamma)^\cdot(s)=\frac{1}{2\lambda(s)}\left(Zf_I(\gamma(s))\dot{\gamma_I}(s)+\overline{Z}f_I(\gamma(s))\dot{\overline{\gamma_I}}(s)\right), \quad \text{for a.e. }s\in(c,d).
$$
\end{proof}
Recalling that $|\dot\gamma(s)|_H=\frac{|\dot\gamma_I(s)|}{2\lambda(s)}$, for an orientation preserving quasiconformal map $f$ we have:
\begin{equation*}
    \left(\frac{|Zf_I(\gamma(s))|-|\overline{Z}f_I(\gamma(s))|}{2f_2(\gamma(s))}\right)|\dot{\gamma}(s)|_H
    \leq
    |(f\circ\gamma)^\cdot(s)|_H
    \leq
    \left(\frac{|Zf_I(\gamma(s))|+|\overline{Z}f_I(\gamma(s))|}{2f_2(\gamma(s))}\right)|\dot{\gamma}(s)|_H
\end{equation*}
for almost every $s$. If a map $f_0$ has the MSP for a family $\Gamma_0$, then by \eqref{MSP-eq} we have equality
\begin{equation*}
   |(f_0\circ\gamma)^\cdot(s)|_H
   =\left(\frac{|Z(f_0)_I(\gamma(s))|-|\overline{Z}(f_0)_I(\gamma(s))|}{2(f_0)_2(\gamma(s))}\right)|\dot{\gamma}(s)|_H\,.
\end{equation*}
The next type of modulus inequality adapts the statement and the proof of Theorem 18 in \cite{BFP-H-MM} to the geometric setting of the affine-additive group.
Before getting into the details, we need to briefly introduce a concept essential for our next arguments. Denote by $B_\Aa(p,r)$ the open ball with respect to the distance $d_\Aa$, centered at $p\in\Aa$ and with radius $r>0$. Let $f:\Omega\to\Aa$ be a quasiconformal map on a domain $\Omega\subseteq\Aa$ and define the volume derivative for $f$ with respect to $\mu_\Aa$ to be the limit
\begin{equation*}
    \mathcal{J}_{\mu_\Aa}(p,f):=\lim_{r\rightarrow0}\frac{\mu_{\Aa}(f(B_\Aa(p,r)))}{\mu_{\Aa}(B_\Aa(p,r))}.
\end{equation*}
The following identity holds:
\begin{equation}\label{Volume derivative}
    \mathcal{J}_{\mu_\Aa}(p,f)=\frac{1}{(2f_2(p))^4}\big(|Zf_I(p)|^2-|\overline Z f_I(p)|^2\big)^2,
\end{equation}
almost everywhere in $\Omega$.\\
For the proof of \eqref{Volume derivative}, see Lemma \ref{Jac WC} in the Appendix.
We can now prove the following:
\begin{prop}\label{quasi-invariance for Mod_4}
    Suppose that $f : \Omega \rightarrow\Omega'$ is a quasiconformal map between two domains in $\Aa$
and $\Gamma$ is a family of curves in $\Omega$. Then
\begin{equation}\label{useful}
    \Mod_4(f(\Gamma))\leq\int_\Omega K(p,f)^2 \rho(p)^4\,d\mu_{\Aa}(p) \text{ for all } \rho\in\Adm(\Gamma).
\end{equation}
\end{prop}
\begin{proof}
Let $\Gamma_0$ be the family of all rectifiable curves in $\Gamma$ on which $f$ is absolutely continuous (the non-rectifiable curves have modulus zero). Since $f$ is quasiconformal, we have $\Mod_4(\Gamma)=\Mod_4(\Gamma_0)$.
Throughout the proof we shall assume $f$ to be differentiable in the sense of \cite{MM} on curves in $\Gamma_0$ almost everywhere on their domain of definition, for otherwise one can consider the argument in the beginning of the proof on Theorem 18 in \cite{BFP-H-MM}.\\ 
We now take an arbitrary admissible density $\tilde\rho\in\Adm(f(\Gamma))$ and we assign to it a pull-back density $\rho_{\tilde\rho}$ defined by
\begin{equation*}
    \rho_{\tilde\rho}(p)=\left\{\begin{matrix}
    \tilde\rho(f(p))\frac{|Zf_I(p)|+|\overline Z f_I(p)|}{2f_2(p)},& & p\in\Omega\\
    \\
    0,& &p\in\Aa\setminus\Omega.
  \end{matrix}\right.
\end{equation*}
We will show that $\rho_{\tilde\rho}$ is admissible for $\Gamma_0$. To this end, let $\gamma:[a,b]\rightarrow\Omega$ be an arbitrary curve in $\Gamma_0$. By definition of $\Gamma_0$, it is rectifiable and therefore it has a parametrization by arc-length, $\tilde\gamma=(\tilde a,\tilde \lambda+i\tilde t):[0,\ell(\gamma)]\rightarrow\Omega$. We know that $f$ is quasiconformal and that $\tilde\gamma(s)$ is horizontal; due to Lemma \ref{chain rule} this reasoning leads to 
\begin{equation*}
|(f\circ\tilde\gamma)^\cdot(s)|_H=\frac{1}{2f_2(\tilde\gamma(s))}\left|Zf_I(\tilde\gamma(s))\,\frac{\dot{\tilde{\gamma}}_I(s)}{2\tilde \lambda(s)}+\overline{Z}f_I(\tilde\gamma(s)) \,\frac{\dot{\overline{\tilde{\gamma}}}_I(s)}{2\tilde \lambda(s)}\right| \text{ for a.e. } s\in[0,\ell(\gamma)]
\end{equation*}
which gives
\begin{equation*}
|(f\circ\tilde\gamma)^\cdot(s)|_H\leq\frac{|Zf_I(\tilde\gamma(s))|+|\overline{Z}f_I(\tilde\gamma(s))|}{2f_2(\tilde\gamma(s))}
|\dot{\tilde{\gamma}}(s)|_H \text{ for a.e. } s\in[0,\ell(\gamma)].
\end{equation*}
We notice that $f\circ\tilde\gamma$ is absolutely continuous and the latter inequality yields
\begin{align*}
    \int_\gamma \rho_{\tilde\rho} \,d\ell=& \int_0^{\ell(\gamma)} \rho_{\tilde\rho}(\tilde\gamma(s))|\dot{\tilde{\gamma}}(s)|_H\,ds\\
    =& \int_0^{\ell(\gamma)} \tilde\rho(f(\tilde\gamma(s)))\frac{|Zf_I(\tilde\gamma(s))|+|\overline{Z}f_I(\tilde\gamma(s))|}{2f_2(\tilde\gamma(s))}|\dot{\tilde{\gamma}}(s)|_H\,ds\\
    \geq& \int_0^{\ell(\gamma)} \tilde\rho(f(\tilde\gamma(s)))|(f\circ\tilde\gamma)^\cdot(s)|_H\,ds=\int_{f\circ\tilde\gamma}\tilde\rho\,d\ell=\int_{f\circ\gamma}\tilde\rho\,d\ell\geq1.
\end{align*}
We deduce that $\rho_{\tilde\rho}\in\Adm(\Gamma_0)$. Making use of \eqref{Volume derivative}, the previous fact allows us to conclude as follows:
\begin{align*}
    \Mod_4(\Gamma_0)=&\inf_{\rho\in\Adm(\Gamma_0)}\int_\Omega \rho^4(p)\,d\mu_{\Aa}(p)\\
                    \leq&\int_\Omega \rho^4_{\tilde\rho}(p)\,d\mu_{\Aa}(p)=\int_\Omega \tilde\rho^4(f(p))\frac{(|Zf_I(p)|+|\overline Z f_I(p)|)^4}{(2f_2(p))^4}\,d\mu_{\Aa}(p)\\
                    =&\int_\Omega \tilde\rho^4(f(p))\left(\frac{|Zf_I(p)|+|\overline Zf_I(p)|}{|Zf_I(p)|-|\overline Zf_I(p)|}\right)^2\frac{\left(|Zf_I(p)|^2-|\overline Z f_I(p)|^2\right)^2}{(2f_2(p))^4}\,d\mu_{\Aa}(p)\\
                    =&\int_\Omega \tilde\rho^4(f(p)) K^2(p,f)\mathcal{J}_{\mu_\Aa}(p,f)\,d\mu_{\Aa}(p)\\
                    =&  \int_{\Omega'} \tilde\rho^4(q) K^2(f^{-1}(q),f)\,d\mu_{\Aa}(q),
\end{align*}
for all $\tilde\rho\in\Adm(f(\Gamma))$. 
We may apply the previous inequality to the quasiconformal map $f^{-1}$ and the curve family $f(\Gamma)$. Thus
\begin{equation}\label{mod ineq 1}
    \Mod_4(f(\Gamma))\leq\int_\Omega K^2(f(p),f^{-1}) \rho^4(p)\,d\mu_{\Aa}(p), 
\end{equation} 
for all $\rho\in\Adm(\Gamma)$. It is straightforward to verify the identity
\begin{equation}\label{distortion quotient identity}
    K(p,f)=K(f(p),f^{-1}) \text{ a. e. in }\Omega.
\end{equation}
Combining \eqref{mod ineq 1} and \eqref{distortion quotient identity} we obtain the desired result.
\end{proof}

\begin{rem}
    We want to underline two other important consequences following from the proof of Proposition \ref{quasi-invariance for Mod_4}. In the first place, we have that
    \begin{equation*}
    \Mod_4(\Gamma)\leq\int_{\Omega'} K(f^{-1}(q),f)^2 \tilde\rho(q)^4\,d\mu_{\Aa}(q) \text{ for all } \tilde\rho\in\Adm(f(\Gamma)),
\end{equation*}
and thus 
\begin{equation}\label{qi 4-Mod}
    \frac{1}{K^2_f}\Mod_4(\Gamma)\leq\Mod_4(f(\Gamma))\leq K^2_f\, \Mod_4(\Gamma).
\end{equation}
\end{rem}
Secondly, we have the following proposition in which conditions both on the foliation of the domain as well as on the quasiconformal map are being set in order to obtain the equality in the modulus inequality for the mean distortion. 
\begin{prop}\label{P Mod4 f Gamma_0}
     Let $f_0:\Omega\rightarrow\Omega'$ be an orientation preserving quasiconformal map between domains in the affine-additive group. As described above, let $\gamma$ be the foliation of $\Omega$ and let $\Gamma_0$ be the curve family. Assume further that $f_0$ has the MSP for $\Gamma_0$ and that 
\begin{equation}\label{K const along fol - strategy}
    K(\gamma(s,\delta),f_0)\equiv K_{f_0}(\delta)
\end{equation}      
for all $(s,\delta)\in(c,d)\times \Delta$. Then 
\begin{equation}
    \Mod_4(f_0(\Gamma_0))=\frac{1}{(d-c)^{3}}\int_\Delta K^2_{f_0}(\delta)\,d\nu(\delta)\,.
\end{equation}
\end{prop}
\begin{proof}
Let $\rho'\in \text{Adm}(f_0(\Gamma_0))$ be an arbitrary density. Since $f_0$ is quasiconformal, we know that the image under $f_0$ of an horizontal curve $\gamma(\cdot,\delta)\in\Gamma_0$ is still a horizontal curve up to a sub-family of curves contained in $\Gamma_0$ with vanishing $4$-modulus. Thanks to the minimal stretching property of $f_0$ we find
\begin{align*}
    1\leq&\int_c^d\rho'(f_0\circ \gamma(s,\delta))|{(f_0\circ\gamma)}^\cdot(s,\delta)|_H\,ds\\
    =&\int_c^d \rho'(f_0\circ \gamma(s,\delta))\frac{\big|Z(f_0)_I(\gamma(s,\delta))\big|-\big|\overline{Z}(f_0)_I(\gamma(s,\delta))\big|}{2(f_0)_2(\gamma(s,\delta))}|\dot\gamma(s,\delta)|_H\,ds.
\end{align*} 
We apply H\"older's inequality with conjugated exponents $4$ and $\frac{4}{3}$ to the last relation; this gives
\begin{equation*}
    \frac{1}{(d-c)^3}\leq \int_c^d \rho'^4(f_0\circ \gamma(s,\delta))\Big(\frac{\big|Z(f_0)_I(\gamma(s,\delta))\big|-\big|\overline{Z}(f_0)_I(\gamma(s,\delta))\big|}{2(f_0)_2(\gamma(s,\delta))}\Big)^4|\dot\gamma(s,\delta)|^4_H\,ds.
\end{equation*}
Now, multiplying both sides by $K^2_{f_0}(\delta)$ and integrating over $\Delta$ with respect to $d\nu$ gives
\begin{align}\label{rel K_f0 1}
    &\frac{1}{(d-c)^3}\int_\Delta K^2_{f_0}(\delta)\,d\nu(\delta)\leq\\\nonumber&\int_\Delta \int_c^d \rho'^4(f_0\circ \gamma(s,\delta))K^2_{f_0}(\delta)\left(\frac{\big|Z(f_0)_I(\gamma(s,\delta))\big|-\big|\overline{Z}(f_0)_I(\gamma(s,\delta))\big|}{2(f_0)_2(\gamma(s,\delta))}\right)^4|\dot\gamma(s,\delta)|^4_H\,ds\,d\nu(\delta).
\end{align}
By plugging in the assumption $K_{f_0}(\delta)\equiv K(\gamma(s,\delta),f_0)$ for all $(s,\delta)\in(c,d)\times\Delta$ into \eqref{Volume derivative}, we obtain the identity  
$$
K^2(\gamma(s,\delta),f_0)\left(\frac{\big|Z(f_0)_I(\gamma(s,\delta))\big|-\big|\overline{Z}(f_0)_I(\gamma(s,\delta))\big|}{2(f_0)_2(\gamma(s,\delta))}\right)^4=\mathcal{J}_{\mu_\Aa}(\gamma(s,\delta),f_0).
$$ 
By recomposing the left-invariant Haar measure on $\Aa$ through the foliation $\gamma$ as 
$$
|\dot\gamma(s,\delta)|^4_H\,ds\,d\nu(\delta)=d\mu_\Aa(p),\;p=\gamma(s,\delta)\in\Omega,
$$ we obtain that \eqref{rel K_f0 1} results into
\begin{equation}\label{ rel K_g 2}
    \frac{1}{(d-c)^3}\int_\Delta K^2_{f_0}(\delta)\,d\nu(\delta)\leq \int_\Omega \rho'^4(f_0(p))\mathcal{J}_{\mu_\Aa}(p,f_0)\,d\mu_{\Aa}(p).
\end{equation} 
We then change of variable $f_0(p)=q\in\Omega'$; then \eqref{ rel K_g 2} becomes 
\begin{equation*}
    \frac{1}{(d-c)^3}\int_\Delta K^2_{f_0}(\delta)\,d\nu(\delta)\leq \int_{\Omega'}\rho'^4(q)\,d\mu_{\Aa}(q).
\end{equation*}
Since $\rho'$ was chosen arbitrarily among the admissible densities of $f_0(\Gamma_0)$, the latter inequality shows that
\begin{equation*}
    \frac{1}{(d-c)^3}\int_\Delta K^2_{f_0}(\delta)\,d\nu(\delta)\leq \Mod_4(f_0(\Gamma_0)).
\end{equation*}
For the other inequality consider the push-forward density given by 
\begin{equation*}
    \rho'_0(q)=\left\{\begin{matrix}
    \frac{2(f_0)_2(\gamma(s,\delta))}{(d-c)|\dot\gamma(s,\delta)|_H\left(|Z(f_0)_I(\gamma(s,\delta))|-|\overline Z (f_0)_I(\gamma(s,\delta))|\right)},& & q=f_0(\gamma(s,\delta))\in\Omega',\\
    \\
    0, & &q\notin\Omega'.
  \end{matrix}\right.
\end{equation*}
Thanks to the minimal stretching property of $f_0$ this density is admissible, i.e. $\rho'_0\in\text{Adm}(f_0(\Gamma_0))$:
\begin{align*}
    \int_{f_0\circ\gamma} \rho'_0\,d\ell=\int_c^d\rho_0'(f_0\circ\gamma(s,\delta))|(f_0\circ\gamma)^\cdot(s,\delta)|_H\,ds=\int_c^d\frac{1}{d-c}\,ds=1.
\end{align*} 
Therefore, via the change of variable $f_0(\gamma(s,\delta))=q\in\Omega'$ for some $(s,\delta)\in(c,d)\times \Delta$, we obtain:
\begin{align*}
M_4(f_0(\Gamma_0))\leq&\int_{\Omega'}\rho'^{\,4}_0(q)\,d\mu_{\Aa}(q)\\
                =&\int_\Omega \rho'^{\,4}_0(f_0(p))\mathcal{J}_{\mu_\Aa}(p,f_0)\,d\mu_{\Aa}(p)\\
                =&\int_\Omega \rho'^{\,4}_0(f_0(p))\frac{1}{(2(f_0)_2(p))^4}\big(|Z(f_0)_I|^2-|\overline Z (f_0)_I|^2\big)^2(p)\,d\mu_{\Aa}(p)\\
                =&\int_\Delta\int_c^d\frac{1}{(d-c)^4|\dot\gamma(s,\delta)|_H^4}K^2(\gamma(s,\delta),f_0)|\dot\gamma(s,\delta)|_H^4\,ds\,d\nu(\delta)\\
                =&\frac{1}{(d-c)^4}\int_\Delta\int_c^d K^2(\gamma(s,\delta),f_0)\,ds\,d\nu(\delta)\\
                =&\frac{1}{(d-c)^3}\int_\Delta K^2_{f_0}(\delta)\,d\nu(\delta).
\end{align*}
In this way the proof is concluded.
\end{proof}
\noindent{\textit{Proof of Thereom \ref{Thm1}}}
Proposition \ref{P Mod4 f Gamma_0}, combined with Proposition \ref{quasi-invariance for Mod_4} prove the statement of Theorem \ref{Thm1}.\qed
\begin{rem}\label{strategy}
    From the statement of Theorem \ref{Thm1} we can develop a method which verifies if a candidate quasiconformal map $f_0:\Omega\to\Omega'$ between domains $\Omega,\Omega'\subset\Aa$, is a minimizer for the mean distortion functional. We describe the steps of method:
    \begin{enumerate}[1.]
        \item let $\mathcal{F}$ be a class of quasiconformal mappings $f:\Omega\to\Omega',\,f\in\mathcal{F}$;
        \item let $\gamma$ be a foliation for $\Omega$ which is composed of horizontal curves decomposing the volume measure of $\Aa$, i.e.  $\gamma$ verifies the assumptions of Proposition \ref{P Mod_4 Gamma_0};
        \item introduce the curve family $\Gamma_0$ and then calculate the extremal density $\rho_0$ for $\Mod_4(\Gamma_0)$ given by \eqref{rho_0 strategy};
        \item verify that the distortion quotient $K(\cdot,f)$ is constant along such horizontal curves foliating $\Omega$, i.e. condition \eqref{K const along fol - strategy};
        \item check the MSP for $f_0$ with respect to $\Gamma_0$;
        \item determine a curve family $\Gamma\supset\Gamma_0$ such that $\Mod_4(f_0(\Gamma_0))\leq\Mod_4(f(\Gamma))$ for all $f\in\mathcal{F}$ and verify $\rho_0\in\Adm(\Gamma)$.
    \end{enumerate}
\end{rem}


\section{The linear stretch map}\label{linear stretch}
For $k>0$, the map $f_k:\Aa\to\Aa$ given by 
\begin{equation}\label{f_k exmp}
    f_k(a,\lambda+it)=(ka,\lambda+ikt),
\end{equation}
shall be called \textit{linear stretch map}. Its name is justified by the fact that it is a linear map with respect to the cartesian coordinates.\\
We will present two geometric settings where the linear stretch map turns to be a minimizer for the mean distortion functional, respectively for $k\in(0,1)$ and for $k>1$. This distinction is motivated by the Beltrami coefficient $\mu_{f_k}=\frac{1-k}{1+k}$: the two geometric settings have distinct suitable domains, foliations and associated curve families such that the MSP for $f_k$ holds for both cases. 
\subsection{The case $k\in(0,1)$.}
    Let $k\in(0,1)$ and define two domains as follows:
    \begin{align*}
        \Omega=&\left\{\left(a+\frac{t}{2\lambda},\lambda+it\right)\in\Aa:a\in(0,1),\lambda\in\left(\frac{1}{2},1\right),t\in(0,1)\right\},\\
        \Omega^k=&\left\{\left(k\left(a+\frac{t}{2\lambda}\right),\lambda+ikt\right)\in\Aa:a\in(0,1),\lambda\in\left(\frac{1}{2},1\right),t\in(0,1)\right\}.
    \end{align*}
For a fixed $t\in\R$, we denote 
$$
\partial\Omega_t=\left\{\left(a+\frac{t}{2\lambda},\lambda+it\right)\in\Aa:a\in(0,1),\lambda\in\left(\frac{1}{2},1\right)\right\}
$$
and consider the class $\mathcal{F}_k$ of all quasiconformal mappings $f:\Omega\to{\Omega^k}$ which extend homeomorphically to the boundary and we impose conditions
$$
f(\partial\Omega_0)=\partial\Omega^k_0
\text{ and }
f(\partial\Omega_1)=\partial\Omega^k_1.
$$
The domain $\Omega$ is displayed in the following figure: 
\begin{figure}[ht!]
\centering
\includegraphics[width=1\textwidth]{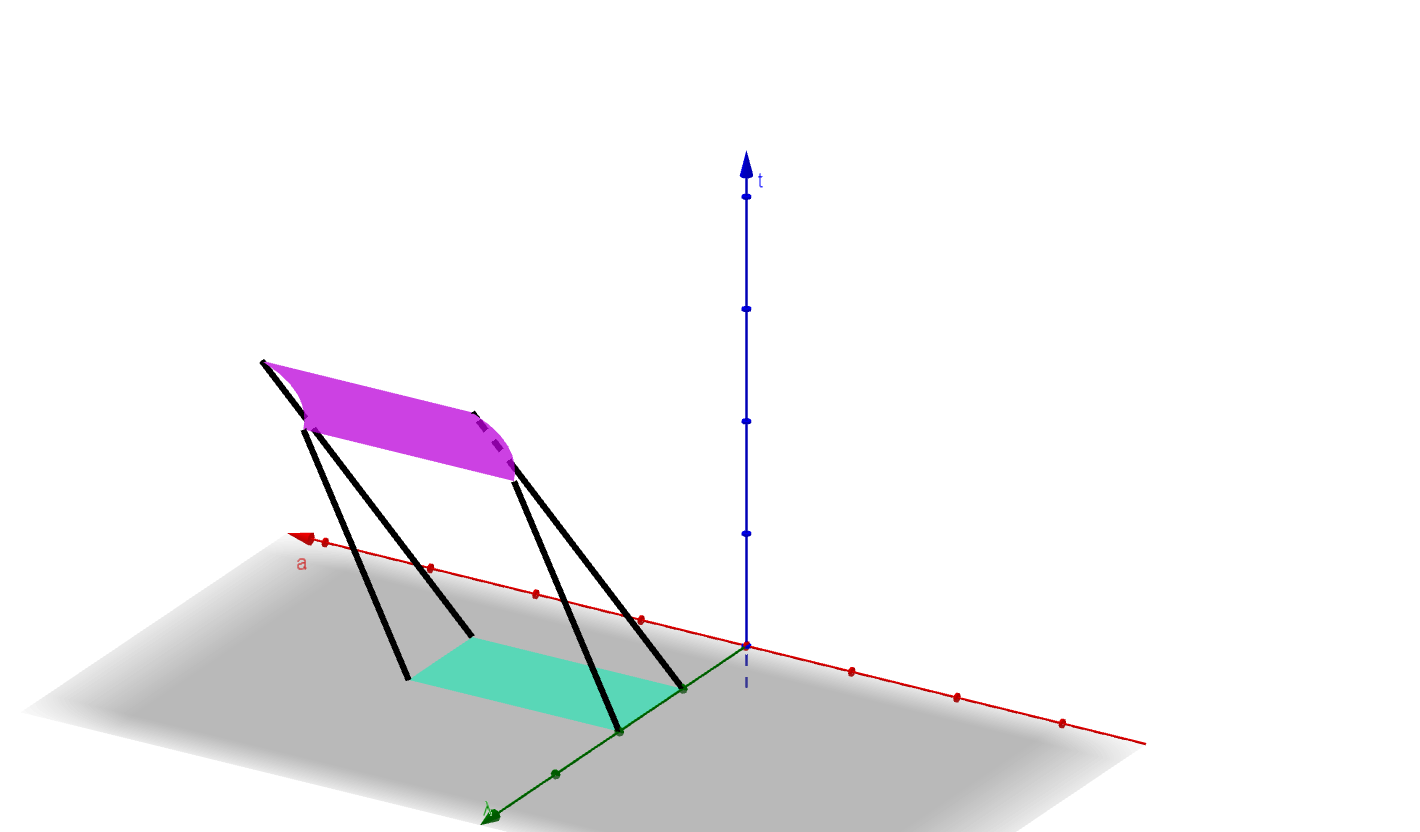}
\caption{$\partial\Omega_0$ is in cyan and $\partial\Omega_1$ is in purple.}
\end{figure}

\newpage
\noindent{\textit{Proof of Theorem \ref{Thm1.5}}}
The steps of the proof steps are the ones explained in Remark \ref{strategy}.\\
\noindent{1. } The class $\mathcal{F}_k$ is presented above the proof.\\
\noindent{2. }Let the pair $(a,\lambda)\in(0,1)\times\left(\frac{1}{2},1\right)$ and let  $\gamma:(0,1)\times(0,1)\times\left(\frac{1}{2},1\right)\to \Omega$ be the foliation of $\Omega$ given by
$$
\gamma(s,a,\lambda)=\left(a+\frac{s}{2\lambda},\lambda+is\right),\quad (s,a,\lambda)\in(0,1)\times(0,1)\times\left(\frac{1}{2},1\right).
$$
In this way, the volume element on $\Aa$ can be written as
$$
d\mu_{\Aa}(\gamma(s,a,\lambda))=\frac{1}{\lambda^2}\,da\,d\lambda\,ds=|\dot \gamma(s,a,\lambda)|^4_H \,ds\,d\nu(a,\lambda),
$$
where
$$
d\nu(a,\lambda)=2^4 \lambda^2 \,da\,d\lambda.
$$
\noindent{3. }In order to apply Proposition \ref{P Mod_4 Gamma_0}, we consider the family of horizontal curves 
$$
\Gamma_0=\left\{\gamma(\cdot,a,\lambda):a\in(0,1),\lambda\in\left(0,\frac{1}{2}\right)\right\}.
$$
An extremal density for $\Gamma_0$ is given by formula \eqref{rho_0}: namely, $\rho_0(a,\lambda+it)=2\lambda\cdot\mathcal{X}_{\Omega}(a,\lambda+it)$. Hence
$$\Mod_4(\Gamma_0)=\int_\frac{1}{2}^1\int_0^1 2^4 \lambda^2\,da\,d\lambda=\frac{14}{3}.$$
\noindent{4. }An explicit calculation gives constant distortion quotient $K_{f_k}$. Indeed,
$$
K(\gamma(s,a,\lambda),f_k)\equiv\frac{1}{k},\quad(s,a,\lambda)\in(0,1)\times(0,1)\times\left(\frac{1}{2},1\right),
$$
and so $K_{f_k}=\esssup_{p}K(p,f_k)=\frac{1}{k}$.\\
\noindent{5. }We observe that $f_k$ has the MSP with respect to the curve family $\Gamma_0$.    
Indeed, for $k\in(0,1)$ we have
$$
\mu_{f_0}(\gamma_{a,\lambda}(s))\frac{\dot{\overline{(\gamma_{a,\lambda})_I}}(s)}{\dot{(\gamma_{a,\lambda})_I}(s)}=\frac{k-1}{1+k}<0\quad \text{for all } s\in(0,1). $$
\noindent{6. } Aiming to apply Theorem \ref{Thm1}, we need to find a bigger curve family $\Gamma\supseteq\Gamma_0$ for which $\rho_0$ is still admissible and such that $\Mod_4(f_k(\Gamma_0))\leq\Mod_4(f(\Gamma))$ for all $f\in\mathcal{F}_k$. A guess for $\Gamma$ is the family of $all$ horizontal curves contained in $\Omega$ which are joining the two components $\partial\Omega_0$ and $\partial\Omega_1$. The boundary conditions for maps in the class $\mathcal{F}_k$ provide that the image $f_k(\Gamma)$ is going to be a family of the same type in $\Omega^k$. Using the absolute continuity of quasiconformal mappings on almost every curve up to a negligible family of curves with zero $4$-modulus and using the boundary conditions, we may show that 
$$
\Mod_4(f_k(\Gamma_0))\leq\Mod_4(f(\Gamma))\quad \text{for all } f\in\mathcal{F}_k
$$
We have to check that $\rho_0$ is admissible for the extended family $\Gamma$.
Indeed, for a curve $\gamma:[c,d]\to\Aa$ with $\gamma\in\Gamma$, we have
\begin{align*}
    \int_\gamma\rho_0\,d\ell=& \int_c^{d}\sqrt{\dot\lambda(s)^2+\dot t(s)^2}\,ds\geq
    \int_c^{d}\dot t(s)\,ds=1
\end{align*}
Here we have used for the evaluation of the integral the fact that $s\mapsto t(s)$ is an absolutely continuous function and the conditions $\gamma(c)\in \partial\Omega_0$, $\gamma(d)\in \partial\Omega_1$.\\
We conclude that $\rho_0\in\Adm(\Gamma)$ and, from Theorem \ref{Thm1}, it follows that
\begin{equation}\label{fun ineq linear strect 1}
K_{f_k}^2\int_{\Omega}\rho_0(p)^4\,d\mu_\Aa(p)\leq\int_{\Omega}K(p,f)^2\rho_0(p)^4\,d\mu_\Aa(p)\quad \text{for all } f\in\mathcal{F}_k\,.
\end{equation}
The proof is complete. \qed

\medskip

We wish to highlight an important consequence which follows from the last proof: by taking the essential supremum for $K^2(\cdot,f)$ on the r.h.s. of \eqref{fun ineq linear strect 1}, we notice that $f_k$ minimizes also the maximal distortion $K_{f}$ (see \eqref{K_f}). We therefore state:
\begin{cor}\label{minim max dist}
      The linear stretch map $f_k:\Omega\to\Omega^k$ is an orientation preserving quasiconformal map such that
    $$
    \quad K_{f_k}\leq K_f,
    $$
    for all $f\in{\mathcal F}_k$\,.   
\end{cor}
\subsection{The case $k>1$.}
 Let $k>1$ and by using a similar notation as in the previous section we consider two domains as follows:
    \begin{align*}
        \Omega=&\left\{\left(a,\lambda+it\right)\in\Aa:a\in(0,1),\lambda\in\left(\frac{1}{2},1\right),t\in(0,1)\right\},\\
        \Omega^k=&\left\{\left(ka,\lambda+ikt\right)\in\Aa:a\in(0,1),\lambda\in\left(\frac{1}{2},1\right),t\in(0,1)\right\}.
    \end{align*}
For a fixed $\lambda>0$, let
$$
\partial\Omega_\lambda=\{\left(a,\lambda+it\right)\in\Aa:a\in(0,1),t\in\left(0,1\right)\}
$$
and consider the class $\mathcal{F}_k$ of all quasiconformal mappings $f:\Omega\to{\Omega^k}$ which extend homeomorphically to the boundary subject to the conditions
$$
f(\partial\Omega_{\frac{1}{2}})=\partial\Omega^k_\frac{1}{2}
\text{ and }
f(\partial\Omega_1)=\partial\Omega^k_1.
$$
\noindent{\textit{Proof of Theorem \ref{Thm1.5}}}
Again, the steps of the proof follow the strategy of Remark \ref{strategy}.\\
\noindent{1. } The class $\mathcal{F}_k$ is as above.\\
\noindent{2. } Let the pair $(a,t)\in(0,1)\times\left(0,1\right)$ and let  $\gamma:\left(\frac{3}{2^{\frac{1}{3}}},3\right)\times(0,1)\times(0,1)\to \Omega$ be the foliation of $\Omega$ given by
$$
\gamma(s,a,t)=\left(a,\frac{s^3}{3^3}+it\right),\quad (s,a,t)\in\left(\frac{3}{2^{\frac{1}{3}}},3\right)\times(0,1)\times(0,1).
$$
In this way, the volume element on $\Aa$ can be written as
$$
d\mu_{\Aa}(\gamma(s,a,t))=\frac{3^4}{s^4}\,da\,ds\,dt=|\dot \gamma(s,a,t)|^4_H \,ds\,d\nu(a,t),
$$
where
$$
d\nu(a,t)=2^4 \,da\,dt.
$$
\noindent{3. } In order to apply Proposition \ref{P Mod_4 Gamma_0}, we consider the family of horizontal curves 
$$
\Gamma_0=\{\gamma(\cdot,a,t):a\in(0,1),t\in\left(0,1\right)\}.
$$
The extremal density $\rho_0$ for $\Gamma_0$ following from formula \eqref{rho_0}, is given by 
$$
\rho_0(a,\lambda+it)=c_0\lambda^\frac{1}{3}\cdot\mathcal{X}_{\Omega}(a,\lambda+it),
$$ 
where $c_0=\frac{2^\frac{4}{3}}{3\left(2^\frac{1}{3}-1\right)}$.
This results into
$$\Mod_4(\Gamma_0)=\frac{1}{\left(3-\frac{3}{2^\frac{1}{3}}\right)^3}\int_0^1\int_0^1 2^4\,da\,d\lambda=\frac{2^5}{3^3 \left(2^{\frac{1}{3}}-1\right)}.$$
\noindent{4. }An explicit calculation gives constant distortion quotient $K_{f_k}$, indeed
$$
K(\gamma(s,a,t),f_k)\equiv k,\quad(s,a,t)\in\left(\frac{3}{2^{\frac{1}{3}}},3\right)\times(0,1)\times(0,1),
$$ 
and so $K_{f_k}=\esssup_{p}K(p,f_k)=k$.\\
\noindent{5. } We observe that $f_k$ has the MSP with respect to the curve family $\Gamma_0$.    
Indeed, for $k>1$ we have
$$
\mu_{f_0}(\gamma_{a,t}(s))\frac{\dot{\overline{(\gamma_{a,t})_I}}(s)}{\dot{(\gamma_{a,t})_I}(s)}=\frac{1-k}{1+k}<0,\quad \text{for all } s\in\left(\frac{3}{2^{\frac{1}{3}}},3\right).
$$ 
\noindent{6. } Now, in order to apply Theorem \ref{Thm1}, we need to find a bigger curve family $\Gamma\supseteq\Gamma_0$ for which $\rho_0$ is still admissible and such that $\Mod_4(f_k(\Gamma_0))\leq\Mod_4(f(\Gamma))$ for all $f\in\mathcal{F}_k$. As in the previous proof, a guess for $\Gamma$ is the family of $all$ horizontal curves contained in $\Omega$ which are joining the two components $\partial\Omega_\frac{1}{2}$ and $\partial\Omega_1$. Using similar arguments as in the previous proof we have
$$
\Mod_4(f_k(\Gamma_0))\leq\Mod_4(f(\Gamma))\quad \text{for all } f\in\mathcal{F}_k.
$$
We have to check that $\rho_0$ is admissible for the extended family $\Gamma$.
Indeed, for a curve $\gamma:[c,d]\to\Aa$ with $\gamma\in\Gamma$, we have
\begin{align*}
    \int_\gamma\rho_0\,d\ell=& c_0\int_c^{d}\lambda(s)^\frac{1}{3}\frac{\sqrt{\dot\lambda(s)^2+\dot t(s)^2}}{2\lambda(s)}\,ds\\
    \geq&
    \frac{c_0}{2}\int_c^{d}\frac{\dot \lambda(s)}{\lambda(s)^\frac{2}{3}}\,ds=\frac{3c_0}{2}\left(\lambda(d)^\frac{1}{3}-\lambda(c)^\frac{1}{3}\right)=1.
\end{align*}
Here we have used for the evaluation of the integral the fact that $s\mapsto \lambda(s)$ is an absolutely continuous function and the conditions $\gamma(c)\in \partial\Omega_{\frac{1}{2}}$, $\gamma(d)\in \partial\Omega_1$.\\
We conclude that $\rho_0\in\Adm(\Gamma)$ and, from Theorem \ref{Thm1}, it follows that
$$
K_{f_k}^2\int_{\Omega}\rho_0(p)^4\,d\mu_\Aa(p)\leq\int_{\Omega}K(p,f)^2\rho_0(p)^4\,d\mu_\Aa(p)\quad \text{for all } f\in\mathcal{F}_k\,.
$$
The proof is complete.\qed

\medskip

In the same manner as in Corollary \ref{minim max dist}, we obtain
\begin{cor}
 The linear stretch map $f_k:\Omega\to\Omega^k$ is an orientation preserving quasiconformal map such that
    $$
    \quad K_{f_k}\leq K_f,
    $$
    for all $f\in{\mathcal F}_k$\,.   
\end{cor}

This case may be viewed as a solution to the Gr\"otzsch problem on the setting of the affine-additive group (see also \cite{Ahl54} and \cite{G} for the classical Gr\"otzsch problem on the complex plane, as well as Section 5.2 in \cite{BFP-H-MM} for the analogous Gr\"otzsch problem on the Heisenberg group). 


\section{Cylindrical-logarithmic coordinates}\label{Sec4}
In order to construct radial stretch maps it is convenient to set up an appropriate type of coordinate system on the affine-additive group.  
To this direction, a first step is to consider cylindrical coordinates: recall that $\Aa$ identifies to $\R\times\mathbf{H}^1_\C$, hence the coordinate map for cylindrical coordinates is given by $\mathcal{C}:\R\times\R_{>0}\times\left(-\frac{\pi}{2},\frac{\pi}{2}\right)\to\Aa$ where
$$
\mathcal{C}(a,r,\psi)=(a,r e^{i\psi}), \quad (a,r,\psi)\in\R\times\R_{>0}\times\left(-\frac{\pi}{2},\frac{\pi}{2}\right).
$$
By applying the transformation $\xi\mapsto e^\xi=r>0$, with $\xi\in\R$, the cylidrical-logarithmic coordinates are defined as $\Phi:\R\times\R\times\left(-\frac{\pi}{2},\frac{\pi}{2}\right)\to\Aa$ where
\begin{equation}\label{log coords}
    \Phi(a,\xi,\psi)=\left(a,e^{\xi+i\psi}\right).
\end{equation}
Moreover, the inverse map for this new type of coordinates is explicitly given by
$$
\Phi^{-1}(a,\lambda+it)=\left(a,\frac{\log(\lambda^2+t^2)}{2},\tan^{-1}\left(\frac{t}{\lambda}\right)\right)\in\R\times\R\times\left(-\frac{\pi}{2},\frac{\pi}{2}\right),\quad(a,\lambda+it)\in\Aa.
$$
On the domain 
\begin{equation*}
    \mathbf A:=\R\times\R\times\left(-\frac{\pi}{2},\frac{\pi}{2}\right),
\end{equation*}
the map $\Phi:\mathbf{A}\to\Aa$ is a smooth diffeomorphism with corresponding Jacobian determinant
\begin{equation}\label{Jac det}
    (\det \Phi_*)_{(a,\xi,\psi)}=e^{2\xi}\neq0.
\end{equation}
It follows that for each curve $\gamma : [c, d] \to\Aa$
and
each point $(a, \xi, \psi)= \Phi^{-1}({\gamma(c)})$, there exists a unique curve $\tilde\gamma=\Phi^{-1}\circ\gamma : [c, d] \to\mathbf{A}$
such that $\tilde\gamma(c) = (a, \xi, \psi)$. If $\gamma$ is absolutely continuous in the Euclidean sense, or if it is $C^k$
for a $k \in \N_0$, then $\tilde \gamma$ will be as much regular as $\gamma$.
Further, the same reasoning applies also for continuous mappings from simply connected domains
in $\Aa$. In detail, every mapping $\tilde f:\mathbf{A}\to\mathbf{A}$ yields a well-defined map $f:\Aa\to\Aa$ by setting $f=\Phi\circ\tilde f\circ\Phi^{-1}$.

In what follows, we are going to use only cylindrical-logarithmic coordinates: we will define a quasiconformal map $f$ between domains in the affine-additive group by giving a formula for $\tilde f$. On the other hand, we will still work with $\tilde f$ in the case where this is convenient.\\\\
It turns out that the stretch map has a much neater form in cylindrical-logarithmic coordinates.
\\ 
In what follows we shall give expression for:
\begin{itemize}
    \item the contact condition;
  \item the horizontal vector fields;
  \item the volume and curve integrals;
  \item the Beltrami coefficient;
  \item the MSP condition, 
\end{itemize} 
in these particular coordinates. We adopt the notation 
$$
\tilde f(a,\xi,\psi)=\left(A(a,\xi,\psi),\Xi(a,\xi,\psi),\Psi(a,\xi,\psi)\right).
$$
Also, if $\eta$  is an index running through $a, \xi$ and  $\psi$, we will write 
$A_\eta=\frac{\partial A}{\partial \eta}$ for a given differentiable function $A$.
\subsection{Horizontality, contact condition and the minimal stretching property.}
In order to apply the modulus method, we will need to understand how the horizontality condition transfers on curves in terms of cylindrical-logarithmic coordinates. The following formulas of horizontality and of line integration for curves in $\mathbf{A}$ are useful.
\begin{prop}\label{prop horiz curve}
    A curve $\gamma:[c,d]\to\Aa$ is horizontal if and only if there exists
an absolutely continuous curve 
$$
\tilde\gamma:[c,d]\to\mathbf{A},\quad \tilde\gamma(s)=(a(s),\xi(s),\psi(s)),
$$
with $\Phi\circ\tilde\gamma=\gamma$ and 
\begin{equation}\label{horiz cond}
    \frac{\dot\psi(s)}{2}+\frac{\tan\psi(s)}{2}\dot\xi(s)-\dot a(s)=0 \quad \text{for almost every } s\in[c,d].
\end{equation}
Moreover, for any Borel function $\rho:\Aa\to[0,+\infty]$, we have
\begin{equation}\label{density change}
    \int_\gamma \rho\,d\ell=\int_c^d\rho(\Phi(\tilde\gamma(s)))\frac{\sqrt{\dot\xi(s)^2+\dot\psi(s)^2}}{2\cos\psi(s)}\,ds\,.
\end{equation}
\end{prop}
\begin{proof}
 If $\tilde\gamma:[c,d]\to\mathbf{A}$ is an absolutely continuous curve satisfying $\eqref{horiz cond}$, then
we consider the absolutely continuous curve $\gamma:=\Phi\circ\tilde\gamma$. Conversely, if $\gamma:[c,d] \to\Aa$
is horizontal, we take $\tilde\gamma:[c,d]\to \mathbf{A}$ to be  $\tilde \gamma=\Phi^{-1}\circ\gamma$.\\ 
Now, consider two almost everywhere differentiable
curves $\gamma:[c,d]\to\Aa$ and $\tilde\gamma:[c,d]\to\mathbf{A}$ such that $\Phi \circ \tilde\gamma = \gamma$
for all $s \in [c,d]$. Let $s$ be a point of differentiability in $[c,d]$.
There exists a neighborhood of $s$ where we also have $\Phi \circ \tilde\gamma = \gamma$. Knowing that the local expression of the contact form corresponds to $\eqref{contact form}$, it follows that the condition for a horizontal curve reads as $\eqref{horiz cond}$.
Then we obtain:
\begin{align*}
    |\dot\gamma(s)|_H&=
    \frac{\sqrt{\dot\xi(s)^2+\dot\psi(s)^2}}{2\cos\psi(s)}.
\end{align*}
For such a horizontal curve $\gamma:[c,d]\to\Aa$, the formula for the curve integral follows immediately since 
$\int_\gamma\rho\,d\ell=\int_c^d \rho(\gamma(s))|\dot\gamma(s)|_H\,ds$.
\end{proof}

Now, we are going to describe the contact form and the contact conditions with respect to the cylindrical-logarithmic coordinates. The cartesian coordinates on $\Aa$ can be defined through the diffeomorphism $\Phi$ using coordinates $(a,\xi,\psi)$. The expression of the contact form $\vartheta$ on $\mathbf{A}$ is
\begin{equation}\label{contact form}
    \vartheta=\frac{d\psi}{2}+\frac{\tan\psi}{2}d\xi-da\,.
\end{equation}
\begin{prop}\label{prop tilde contact}
    Let $Q$ be an open set in $\mathbf{A}$ and assume that there exist $C^1$ maps $\tilde f:Q\to\mathbf{A}$ and $f:\Phi(Q)\to\Aa$ such that $f=\Phi\circ\tilde f\circ \Phi^{-1}$ on $Q$. Then the following conditions are equivalent:
    \begin{enumerate}[(1)]
        \item the map $f$ is a contact transformation;
        \item there exists a nowhere vanishing function $\tilde \lambda:Q\to\R$ such that the map $\tilde f=(A,\Xi,\Psi)$ is a $C^1$ diffeomorphism satisfying the system of p.d.e.s  
        \begin{align}\label{contact eqs}
            & \Psi_\psi+\tan\Psi\,\Xi_\psi-2A_\psi=\tilde\lambda  \nonumber\\
            & \Psi_\xi+\tan\Psi\,\Xi_\xi-2A_\xi=\tilde\lambda\tan\psi\\
            & 2A_a-\Psi_a-\tan\Psi\,\Xi_a=2\tilde\lambda\,. \nonumber
        \end{align}
    \end{enumerate}
\end{prop}
\begin{proof}
    Since $f$ and $\tilde f$ are related by $f=\Phi\circ\tilde f\circ \Phi^{-1}$, it is straightfoward to see that $f$ is a $C^1$ diffeomorphism if and only if $\tilde f$ is so.\\
    Now, focusing on the contact conditions, we recall that the map $\Phi$ is a diffeomorphism and that the contact form $\vartheta$ is given by $\eqref{contact form}$. The condition that there exists $\lambda(p)\neq0$ such that $(f^*\vartheta)_p=\lambda(p)\vartheta_p$ is equivalent to $\eqref{contact eqs}$ with $\tilde\lambda=\lambda\circ\Phi$.
\end{proof}

\begin{rem}
    We wish to underline here that a quasiconformal mapping $f$ is differentiable almost everywhere and contact almost everywhere. Thus the corresponding $\tilde f$ satisfies the system of p.d.e.s \eqref{contact eqs} almost everywhere. 
\end{rem}
Below, we give expressions for the vector fields $Z$ and $\overline Z$ in terms of cylindrical-logarithmic coordinates. Straightforward calculations yield
\begin{align}
    Z&=e^{-i\psi}\cos\psi (\partial_\xi-i\partial_\psi)-\frac{i}{2}\partial_a\,,\\
    \overline Z&=e^{i\psi}\cos\psi (\partial_\xi+i\partial_\psi)+\frac{i}{2}\partial_a \,.
    \end{align}
Let $f$ and $\tilde f=(A,\Xi,\Psi)$ be $C^1$ maps as in Proposition \ref{prop tilde contact}. The Beltrami coefficient of $f$ is given by
\begin{equation}\label{Beltrami coeff}
    \mu_f(\Phi(a,\xi,\psi))=\left(\frac{\overline Z(\Xi+i\Psi)}{Z(\Xi+i\Psi)}\right)_{|(a,\xi,\psi)}\,.
\end{equation}
Now assume in addition that $f$ is an orientation preserving quasiconformal map. Let $\tilde\Gamma$ be a family of $C^1$ curves
$$
\tilde\gamma:[a,b]\to\mathbf{A},\quad \tilde\gamma(s)=(a(s),\xi(s),\psi(s))
$$
such that 
\begin{equation*}
    \frac{\dot\psi(s)}{2}+\frac{\tan\psi(s)}{2}\dot\xi(s)-\dot a(s)=0 \quad \text{for all }s\in(a,b),
\end{equation*}
and 
\begin{equation}\label{MSP Phi}
    \frac{\dot\xi(s)-i\dot\psi(s)}{\dot\xi(s)+i\dot\psi(s)}\left(\frac{\overline Z(\Xi+i\Psi)}{Z(\Xi+i\Psi)}\right)_{|\tilde\gamma(s)}<0,
\end{equation}
for $s\in(a,b)$ with $\mu_f(\Phi(\tilde\gamma(s))\neq0.$ Then $f$ has the MSP for the family $\Gamma=\{\Phi\circ\tilde\gamma:\tilde\gamma\in\tilde\Gamma\}$.

\section{The radial stretch map}\label{trunc.td cyl shell}
In this Section we prove Theorem \ref{Thm2}.
In order to construct an analogue of the radial stretch map $z\mapsto |z|^{k-1}z$ in the setting of $\Aa$, we 
detect a suitable domain where this radial stretch map will be defined. This domain happens to be a truncated cylindrical shell, parallel to the $a$-axis of $\Aa$.\\
In detail, for $0<\psi_0<\frac{\pi}{2}\text{ and }r_0>1$ we define
\begin{equation}\label{fol domain} 
D_{r_0,\,\psi_0}=\left\{\left(a+\frac{\tan\psi}{2}\xi,\,e^{\xi+i\psi}\right)\in\Aa:a\in(0,1),\,\psi\in(0,\psi_0),\,\xi\in(0,\log r_0)\right\}.
\end{equation}
Furthermore, we define the following subsets of $\partial D_{r_0,\,\psi_0}$:
 \begin{align}\label{E,F}
     E&=\left\{\left(a, e^{i\psi}\right)\in\Aa:a\in(0,1),\,\psi\in(0,\psi_0)\right\},\\
     F&=\left\{\left(a+\frac{\tan\psi}{2}\log r_0,\,r_0 e^{i\psi}\right)\in\Aa:a\in(0,1),\,\psi\in(0,\psi_0)\right\}.
 \end{align}
 By varying $a\in(0,1)$ and $\psi\in(0,\psi_0)$, we see that $E$ and $F$ are connected by horizontal curves $\gamma_{a,\,\psi}:[0,\log r_0]\to D_{r_0,\psi_0}$ given by
$$
\gamma_{a,\,\psi}(s)=\left(a+\frac{\tan\psi}{2}s,e^{s+i\psi} \right).
$$
Before we proceed, we shall give a volume formula with respect to the logarithmic-cylindrical coordinates and then apply it to the particular case of $D_{r_0,\psi_0}$. 
Let $\Omega\subseteq\Aa$ be a measurable set and let $Q\subseteq\mathbf{A}$ be an open set such that its image $\Phi(Q)=\Omega$. Then a function $h:\Omega\to\R$ is integrable if and only if $(h\circ\Phi)|\det\Phi_*|$ is integrable on $Q$ and in this case we have 
$$
\int_\Omega h(p)\,d\mu_\Aa(p)=\int_Q \frac{h(\Phi(a,\xi,\psi))}{\cos^2\psi}\,d\mathcal{L}^3(a,\xi,\psi)\,.
$$
For every integrable function $h:D_{r_0,\,\psi_0}\to\R$ we have
 $$
 \int_{D_{r_0,\,\psi_0}} h(p)\,d\mu_\Aa(p)=\int_0^{\psi_0}\int_0^{\log r_0}\int_{\frac{\tan\psi}{2}\xi}^{1+\frac{\tan\psi}{2}\xi} \frac{h(\Phi(a,\xi,\psi))}{\cos^2\psi}\,da\,d\xi\,d\psi\,.
 $$
 
$D_{r_0,\psi_0}$ for the case $r_0=e$, $\psi_0=\frac{\pi}{4}$ is in the following figure.
\begin{figure}[ht]
\centering
\includegraphics[width=1\textwidth]{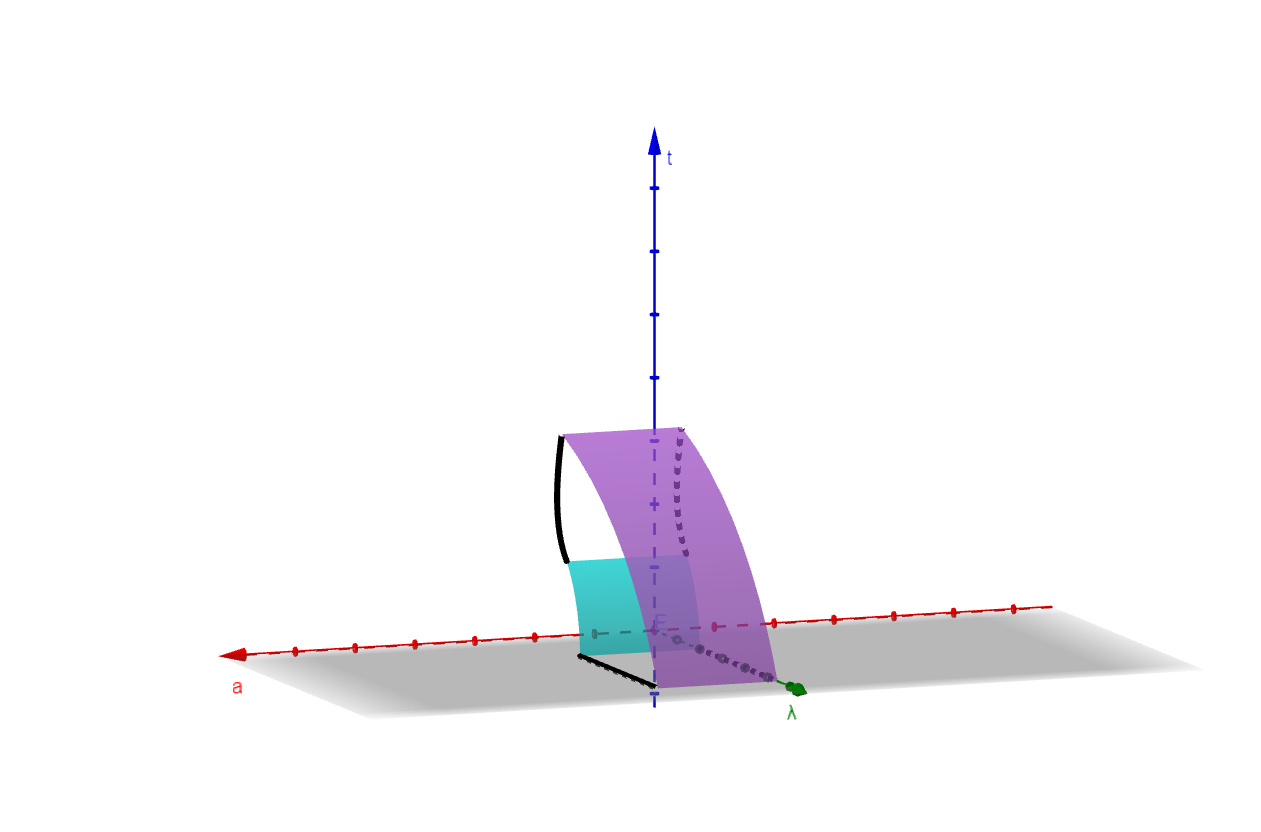}
\caption{Domain $D_{e,\frac{\pi}{4}}$ with $E$ in cyan and $F$ in purple.}
\end{figure}

\subsection{The radial stretch map. Proof of Theorem \ref{Thm2}.}
In this section we construct the radial stretch map on the affine-additive group. We prove Theorem \ref{Thm2} and discuss the properties of the radial stretch map in the remarks.\\
 Let $0<k<1$; 
we start by considering logarithmic-polar coordinates $(\xi,\psi)\in\R\times\left(-\frac{\pi}{2},\frac{\pi}{2}\right)$ on $\textbf{H}^1_\C$ with symplectic form $\omega=\frac{d\xi\wedge d\psi}{4\cos^2\psi}$ and, with respect to the same coordinates, we set the $1$-form $\tau$ on $\textbf{H}^1_\C$ given by $\tau=\frac{d\psi}{2}+\frac{\tan\psi}{2}d\xi$. We introduce the symplectic and planar radial stretch map $g_k:\R\times\left(-\frac{\pi}{2},\frac{\pi}{2}\right)\to\R\times\left(-\frac{\pi}{2},\frac{\pi}{2}\right)$, defined as 
$$
g_k(\xi,\psi)=\left(k\xi,\tan^{-1}\left(\frac{\tan\psi}{k}\right)\right).
$$
Now, let $p=\Phi(a,\xi,\psi)\in\Aa$, take $\gamma$ to be an horizontal path joining $e$ with $p$ and we construct, in cylindrical-logarithmic coordinates, $\tilde{f}_k:\mathbf{A}\to\mathbf{A}$, i.e. the lift of $g_k$ as
\begin{align}\label{the map}
    \tilde{f}_k(a,\xi,\psi)&=\left(\int_{\pi(\gamma)}g^*_k\tau,\,k\xi,\,\tan^{-1}\left(\frac{\tan\psi}{k}\right)\right)\nonumber\\
    &=\left(a-\frac{\psi}{2}+\frac{1}{2}\tan^{-1}\left(\frac{\tan\psi}{k}\right),\,k\xi,\,\tan^{-1}\left(\frac{\tan\psi}{k}\right)\right).
\end{align}
Let $r_0>1$, $0<\psi_0<\frac{\pi}{2}$ and let also the domain $D_{r_0,\psi_0}$. Consider another truncated cylindrical shell $D^k_{r_0,\psi_0}$. In cylindrical-logarithmic coordinates those domains are given by
\begin{align*}
    D_{r_0,\,\psi_0}=&\left\{\Phi\left(a+\frac{\tan\psi}{2}s,s,\psi\right)\in\Aa:a\in(0,1),\,\psi\in(0,\psi_0),\,s\in(0,\log r_0)\right\},\\
    D^k_{r_0,\,\psi_0}=&\left\{\Phi\left(a+\frac{\tan\psi}{2}s-\frac{\psi}{2}+\frac{1}{2}\tan^{-1}\left(\frac{\tan\psi}{k}\right),\,ks,\, \tan^{-1}\left(\frac{\tan\psi}{k}\right)\right)\in\Aa:\right.\\ 
    &\left.\;\;\; a\in(0,1),\,\psi\in(0,\psi_0),\,s\in(0,\log r_0)\right\}.
\end{align*}
Now, we setup a precise boundary condition for a mapping problem. The subsets $E$ and $F$ of $\partial D_{r_0,\,\psi_0}$ (see \eqref{E,F}) are given in cylindrical-logarithmic coordinates by
\begin{align*}
E=&\{\Phi\left(a,0,\psi\right)\in\Aa:a\in(0,1),\psi\in(0,\psi_0)\},\\
F=&\left\{\Phi\left(a+\frac{\tan\psi}{2},1,\psi\right)\in\Aa:a\in(0,1),\psi\in(0,\psi_0)\right\},
\end{align*}
respectively. Also,  we consider the following subsets of $\partial D^k_{r_0,\,\psi_0}$:
\begin{align*}
E^k=&\left\{\Phi\left(a-\frac{\psi}{2}+\frac{1}{2}\tan^{-1}\left(\frac{\tan\psi}{k}\right),0, \tan^{-1}\left(\frac{\tan\psi}{k}\right)\right)\in\Aa:
\, a\in(0,1),\psi\in(0,\psi_0)\right\},\\
F^k=&\left\{\Phi\left(a+\frac{\tan\psi}{2}-\frac{\psi}{2}+\frac{1}{2}\tan^{-1}\left(\frac{\tan\psi}{k}\right),k\log r_0, \tan^{-1}\left(\frac{\tan\psi}{k}\right)\right)\in\Aa:\right.\\
&\left.\;\;\,\,a\in(0,1),\psi\in(0,\psi_0)\right\}.
\end{align*}

Denote by $\mathcal{F}_k$ the class of all quasiconformal maps $\overline{D_{r_0,\,\psi_0}}\to\overline{ D^k_{r_0,\,\psi_0}}$. They map homeomorphically the component $E$ to $E^k$ and the component $F$ to $F^k$, respectively.\\\\
\noindent{\textit{Proof of Theorem \ref{Thm2}.}}
 We prove first that $f_k:D_{r_0,\,\psi_0}\to D^k_{r_0,\,\psi_0}$ is a quasiconformal map. The assumptions of Proposition \ref{prop tilde contact} are satisfied by the smooth map $\tilde{f}_k:\mathbf{A}\to\mathbf{A}$; thus the stretch map $f_{k\,|D_{r_0,\psi_0}}$ is a smooth contact transformation onto its image. Formula \eqref{Beltrami coeff} yields that
$$
\mu_{f_k}(\Phi(a,\xi,\psi))=e^{2i\psi}\frac{k^2-1}{k^2+2\tan^2\psi+1},\quad(a,\xi,\psi)\in\mathbf{A},
$$
proving $\|\mu_{f_k}\|_\infty<1$. 
We have therefore proved that $f_k$ is a smooth orientation preserving quasiconformal map on $D_{r_0,\psi_0}$ with
$$
\|\mu_{f_k}\|_\infty=\frac{1-k^2}{1+k^2}<1\;\text{and}\;K_{f_k}=\frac{1}{k^2}<\infty.
$$
We next prove that 
$$
\int_{D_{r_0,\psi_0}}K^2(p,f_k)\rho_0(p)^4\,d\mu_\Aa(p)\leq\int_{D_{r_0,\psi_0}}K^2(p,f)\rho_0(p)^4\,d\mu_\Aa(p),
$$
with $\rho_0(a,\lambda,t)=(\log r_0)^{-1}\frac{2\lambda}{ |\lambda+it|}$ for all $f\in\mathcal{F}_k$. In order to do so, we once more follow the steps of the proof as in Remark \ref{strategy}.\\
\noindent{1. } The class $\mathcal{F}_k$ is as above.\\
\noindent{2. } Let $\Delta=(0,1)\times(0,\psi_0)$; we define
\begin{equation}\label{foliating diffeo}
    \tilde{\gamma}:(0,\log r_0)\times\Delta\to\mathbf{A}, \quad \tilde{\gamma}(s,a,\psi)=(a(s),\,\xi(s),\,\psi(s))=\left(a+\frac{\tan\psi}{2}s,\,s,\,\psi\right),
\end{equation}
and
$$
\gamma:(0,\log r_0)\times\Delta\to D_{r_0,\psi_0},\quad \gamma(s,a,\psi)=\Phi(\tilde{\gamma}(s,a,\psi)).
$$
The smooth diffeomorphism $\gamma$ has nowhere vanishing Jacobian determinant $\det \gamma_* (s,a,\psi)=e^{2s}$. Further, for each fixed $(a,\psi)\in\Delta$ the curve 
$$
\gamma(\cdot,a,\psi):(0,\log r_0)\to D_{r_0,\psi_0}, \quad s\mapsto\Phi\left(a+\frac{\tan\psi}{2}s,s,\psi\right),
$$
is horizontal: indeed, we observe that
$$
\frac{\dot\psi(s)}{2}+\frac{\tan\psi(s)}{2}\dot\xi(s)-\dot a(s)=0,\quad s\in(0,\log r_0),
$$
and we use Proposition \ref{prop horiz curve}. Additionally, 
$$
|\dot \gamma(s,a,\psi)|_H=\frac{1}{2\cos\psi}\neq0 \quad \text{for all } (s,a,\psi)\in(0,\log r_0)\times\Delta.
$$
In this way, by introducing $\delta=(a,\psi)\in\Delta$, the volume element on $\Aa$ may be written as
$$
d\mu_{\Aa}(\gamma(s,\delta))=\frac{1}{\cos^2\psi}\,ds\,da\,d\psi=|\dot \gamma(s,a,\psi)|^4_H \,ds\,d\nu(a,\psi),
$$
where
$$
d\nu(a,\psi)=2^4 \cos^2\psi\,da\,d\psi.
$$
\noindent{3. } Our model curve family is
$$
\Gamma_0=\{\gamma(\cdot,a,\psi):(a,\psi)\in(0,1)\times(0,\psi_0)\}.
$$
According to Proposition \ref{P Mod_4 Gamma_0}, an extremal density for $\Gamma_0$ is $\rho_0$ defined by
\[
\rho_0(p)=\begin{cases}
    \log(r_0)^{-1}2\cos\psi, &\text{if } p=\gamma(s,a,\psi)\in D_{r_0,\psi_0},\\
    0, &\text{if } p\notin D_{r_0,\psi_0}\,,
\end{cases}
\]
and also
$$
\Mod_4(\Gamma_0)=\left(\frac{2}{\log r_0}\right)^3(\psi_0+\sin\psi_0\cos\psi_0)\,.
$$
\noindent{4. } Since we have 
$$
K(\gamma(s,a,\psi),f_k)=\frac{1}{k^2\cos^2\psi+\sin^2\psi},\, (s,a,\psi)\in(0,\log r_0)\times\Delta,
$$
we notice that the distortion $K(\gamma(s,a,\psi),f_k)$ does not depend on $s\in(0,\log r_0)$, but only on $\psi\in(0,\psi_0)$. This means that the distortion $K(\cdot,f_k)$ is constant along every curve $\gamma$ in the sense of \eqref{K const on curves}. \\
\noindent{5. } We use the criterion given in \eqref{MSP Phi} to verify the MSP for $f_k$ with respect to the curve family $\Gamma_0$. We check straightforwardly that
$$
 \frac{\dot\xi(s)-i\dot\psi(s)}{\dot\xi(s)+i\dot\psi(s)}\left(\frac{\overline Z(\Xi+i\Psi)}{Z(\Xi+i\Psi)}\right)_{|\tilde\gamma(s)}=\frac{k^2-1}{k^2+2\tan^2\psi+1}<0,
$$
for all $s\in(0,\log r_0)$. This holds true for all $k$ such that $0<k<1$. In this way, due to Proposition \ref{P Mod4 f Gamma_0}, we obtain
\begin{align}\label{mod f_k Gamma_0 radial stretch}
    \Mod_4(f_k(\Gamma_0))=&\frac{2^4}{(\log r_0)^3}\int_0^{\psi_0}\int_0^1 K^2_{f_k}(a,\psi)\cos^2\psi\,da\,d\psi\nonumber\\
    =&\frac{2^4}{(\log r_0)^3}\int_0^{\psi_0}\frac{\cos^2\psi}{(k^2\cos^2\psi+\sin^2\psi)^2}\,d\psi\\
    =&\int_{D_{r_0,\psi_0}}K^2(p,f_k)\rho_0^4(p)\,d\mu_{\Aa}(p)\nonumber.
\end{align}
\noindent{6. } We now define a bigger curve family $\Gamma\supseteq\Gamma_0$ for which $\rho_0$ is still admissible and such that $\Mod_4(f_k(\Gamma_0))\leq\Mod_4(f(\Gamma))$ for all $f\in\mathcal{F}_k$. A typical guess for $\Gamma$ is the family of $all$ absolutely continuous and almost everywhere horizontal curves contained in $D_{r_0,\psi_0}$ which are joining the two components $E$ and $F$. The boundary conditions for maps in the class $\mathcal{F}_k$ assure us that the image $f_k(\Gamma)$ is going to be a family of the same type in $D^k_{r_0.\psi_0}$. Using the absolute continuity of quasiconformal mappings on almost every curve up to a negligible family of curves with zero $4$-modulus and using the boundary conditions, we can show that 
\begin{equation}\label{mod ineq f_k}
    \Mod_4(f_k(\Gamma_0))\leq\Mod_4(f(\Gamma))\quad \text{for all } f\in\mathcal{F}_k.
\end{equation}
Eventually, we have to check that $\rho_0$ is admissible for the extended family $\Gamma$.
Observe that from Proposition \ref{prop horiz curve} it follows that for a curve $\gamma:[c,d]\to\Aa$ with $\gamma\in\Gamma$, we have
\begin{align*}
    \int_\gamma\rho_0\,d\ell=&\frac{1}{\log r_0} \int_c^{d}\sqrt{\dot\xi(s)^2+\dot\psi(s)^2}\,ds\geq
    \frac{1}{\log r_0}\int_c^{d}\dot\xi(s)\,ds \\
    =&\frac{1}{\log r_0}(\log r_0-0)=1\,.
\end{align*}
Here, to evaluate the integral we have used the fact that $s\mapsto\xi(s)$ is an absolutely continuous function and the conditions $\gamma(c)\in E$, $\gamma(d)\in F$.\\
We conclude that $\rho_0\in\Adm(\Gamma)$ and from Theorem \ref{Thm1} it follows that
$$
\int_{D_{r_0,\psi_0}}K(p,f_k)^2\rho_0(p)^4\,d\mu_\Aa(p)\leq\int_{D_{r_0,\psi_0}}K(p,f)^2\rho_0(p)^4\,d\mu_\Aa(p)\quad \text{for all } f\in\mathcal{F}_k\,.
$$
The proof is complete.
\qed
\begin{rem}
 It is straightforward to show that the map $f_k$, $k>1$ is quasiconformal with $K_{f_k}=k^2$. Indeed, it is enough to recover the same arguments from the first part of the above proof. However, proving extremality in the case $k>1$ requires a different argument.
\end{rem}
\begin{rem}
In cartesian coordinates, the map $f_k=\Phi\circ\tilde{f}_k\circ\Phi^{-1}:\Aa\to\Aa$, is given by 
$$
f_k(a,\lambda+it)=\left(a-\frac{1}{2}\tan^{-1}\left(\frac{t}{\lambda}\right)+\frac{1}{2}\tan^{-1}\left(\frac{t}{\lambda k}\right), \left(\frac{(\lambda^2+t^2)^k}{\lambda^2 k^2+t^2}\right)^{\frac{1}{2}}\cdot(\lambda k+it)\right).
$$
\end{rem}
\begin{rem}
By writing the coordinate map $\varphi:\R\times\left(-\frac{\pi}{2},\frac{\pi}{2}\right)\to\mathbf{H}^1_\C$ as $\varphi(\xi,\psi)=e^{\xi+i\psi}$, $(\xi,\psi)\in\R\times\left(-\frac{\pi}{2},\frac{\pi}{2}\right)$, and setting $\check{f}_k=\varphi\circ g_k\circ\varphi^{-1}:\mathbf{H}^1_\C\to\mathbf{H}^1_\C$, we notice that the map $f_k:\Aa\to\Aa$ has the lifting property $\pi\circ f_k=\check{f}_k\circ\pi$.
\end{rem}
\begin{rem}
    Making use of the formal substitution $k=-1$, we obtain that the map $f_{-1}:\Aa\to\Aa$ given by
$$
(a,\lambda+it)\mapsto\left(a-\tan^{-1}\left(\frac{t}{\lambda}\right),\frac{-\lambda+it}{|\lambda+it|^2}\right),
$$
is a contactomorphism with $f_{-1}^*\vartheta=\vartheta$ and also a conformal map ($1$-quasiconformal).
\end{rem}

\section{Open question}\label{OpPb}
In this final section we want to discuss the minimality of $f_k$ for the maximal distortion $K_{f_k}$, see \eqref{K_f}. Let $r_0>1$, $\psi_0\in\left(0,\frac{\pi}{2}\right)$ and we make use of the same notation as in the proof of Theorem \ref{Thm2}, reminding that the curve family $\Gamma\supseteq  \Gamma_0$ consists of all
horizontal curves contained in $D_{r_0,\psi_0}$ which connect the two boundary components $E$ and $F$ of $\partial D_{r_0,\psi_0}$. 
By coupling the modulus inequality given in \eqref{mod ineq f_k} with the right inequality in \eqref{qi 4-Mod}, we obtain the chain of inequalities
\begin{equation}\label{mod consequence}
   \frac{\Mod_4 (f_k(\Gamma))}{\Mod_4(\Gamma)}\leq \frac{\Mod_4 (f(\Gamma))}{\Mod_4(\Gamma)}\leq K^2_f,\quad f\in\mathcal{F}_k.
\end{equation}
Now, notice that if we had 
\begin{equation}\label{hope mod}
K^2_{f_k}=\frac{\Mod_4 (f_k(\Gamma))}{\Mod_4(\Gamma)},
\end{equation}
we would conclude
$$K_{f_k}\leq K_f,\quad f\in\mathcal{F}_k.$$
On the other hand, this is not the case because \eqref{hope mod} does not hold for all $\psi_0\in\left(0,\frac{\pi}{2}\right)$.
To this end, we recall that the density $\rho_0$ is still admissible for the larger family $\Gamma\supseteq  \Gamma_0$. This gives
the modulus identity $\Mod_4(\Gamma)=\Mod_4(\Gamma_0)$ and thus
$$
\Mod_4(\Gamma)=\left(\frac{2}{\log r_0}\right)^3(\psi_0+\sin\psi_0\cos\psi_0)\,.
$$ 
To $\rho_0$ we can assign a pushforward density ${f_k}_\#\rho_0$ given by 
$$
{f_k}_\#\rho_0(q)=\left\{\begin{matrix}
    \frac{2(f_k)_2(f_k^{-1}(q))}{|Z(f_k)_I(f_k^{-1}(q))|-|\overline Z(f_k)_I(f_k^{-1}(q))|}\rho_0(f_k^{-1}(q)), &\text{if }& q\in D^k_{r_0,\psi_0},\\
    \\
    0, &\text{if }& q\notin D^k_{r_0,\psi_0}.
\end{matrix}\right.
$$
Based on the proof of Theorem \ref{Thm2}, it is straightforward to see that ${f_k}_\#\rho_0\in\Adm(f_k(\Gamma))$, giving an analogous modulus identity $\Mod_4(f_k(\Gamma))=\Mod_4(f_k(\Gamma_0))$. We refer to \eqref{mod f_k Gamma_0 radial stretch} and this gives  
$$
\Mod_4(f_k(\Gamma))=\left(\frac{2}{\log r_0}\right)^3 k^{-3}\left(\frac{k \sin2\psi_0}{1+k^2+(k^2-1)\cos2\psi_0}+\tan^{-1}\left(\frac{\tan\psi_0}{k}\right)\right).
$$
For $\psi_0\in(\frac{\pi}{4},\frac{\pi}{2})$ we observe that the function
$$
k\mapsto\frac{k^\frac{3}{2} \sin2\psi_0}{1+k^2+(k^2-1)\cos2\psi_0}+k^\frac{1}{2}\tan^{-1}\left(\frac{\tan\psi_0}{k}\right),\quad k\in(0,1),
$$
is monotone increasing (by a direct calculation the derivative is positive for $k\in(0,1)$) and thus bounded from above by $\psi_0+\sin\psi_0\cos\psi_0$.
Therefore
\begin{equation*}
    \frac{\Mod_4(f_k(\Gamma))}{\Mod_4(\Gamma)}
    \leq k^{-\frac{7}{2}}< k^{-4}=K^2_{f_k}.
\end{equation*}
 In the case $\psi_0\in\left(\frac{\pi}{4},\frac{\pi}{2}\right)$ we see that equality does not necessarily hold for the stretch map $f_k$ and the curve family $\Gamma$. Despite the latter inequality holds strictly it could still be true that $f_k$ is minimal for the maximal distortion and it is an open question. 

\section{Appendix}\label{Appendix}
\subsection{Background results on quasiconformal mappings in the affine-additive group.}
\subsubsection{The affine-additive group.}
The \textit{affine-additive group} is a Lie group with underlying manifold $\R\times\mathbf{H}^1_\C$. 
We recall the complex vector fields (shortly CVF) $Z,\overline Z$, defined in \eqref{CVF}, and notice that they satisfy the non-trivial commutator identity $[Z,\overline Z]=(\overline Z-Z)+iW$. The Lie algebra of left invariant vector fields of the affine-additive group admits a grading   $$\text{span}_\R\{\text{Im}Z,\text{Re}Z\}\oplus\text{span}_\R\{W\}.$$
The elements of the first layer are referred as \textit{horizontal left invariant vector fields}. The horizontal bundle $\mathcal{H}_\Aa$ is the subbundle of the tangent bundle $T(\Aa)$ whose fibers are the \textit{horizontal subspaces} 
$$
\mathcal{H}_{p,\,\Aa}=\text{span}_\R\{\text{Im}Z_p,\text{Re}Z_p\},\quad p\in\Aa.
$$
Recall also that the contact form for $\Aa$ is
$
\vartheta=\frac{dt}{2\lambda}-da.
$
A contact transformation $f:\Omega\to\Omega'$ on $\Aa$ is a diffeomorphism between domains $\Omega$ and $\Omega'$ in $\Aa$ which preserves the contact structure, i.e.
\begin{equation}\label{Cont Cond}
    f^*\vartheta=\sigma\vartheta,
\end{equation}
for some non-vanishing smooth function $\sigma:\Aa\to\R$. Through the identification of $\Aa$ with $\R\times\mathbf{H}^1_\C$ we write $f=(f_1,f_I)$, $f_I=f_2+if_3$. A contact map $f$ is determined by the following system of p.d.e.s
\begin{align}\label{contact cond Z, Zbar}
    &Zf_3=2f_2 Zf_1,\nonumber\\
    &\overline Zf_3=2f_2\overline Zf_1,\\
    &Wf_3=2f_2(\sigma+Wf_1)\nonumber.
\end{align}

\subsubsection{Quasiconformal mappings.}
Any smooth and metric quasiconformal map between domains in $\Aa$ is locally a contact transformation; this fact comes as a consequence of Proposition 3.3 in \cite{BBP1} and Theorem 1 in \cite{KR1}. Let $\mathcal{L}^3$ be the three dimensional Lebesgue measure on $\R\times\mathbf{H}^1_\C$: we point out that $\mu_{\Aa}\ll\mathcal{L}^3$ and therefore the terminology "almost everywhere" is well defined on $\Aa$ in the sense of $\mathcal{L}^3$. In general, quasiconformal maps on $\Aa$ do not need to be smooth, but rather they belong to an apposite class of Sobolev mappings and they satisfy the contact conditions almost everywhere. Explicitly, let $\leq p<\infty$ and let $\Omega$ be a domain in $\Aa$. We say that a function $u:\Omega\to\C$ belongs to the \textit{horizontal Sobolev space}, $u\in HW^{1,p}(\Omega,\C)$, if $u\in L^p(\Omega,\C)$ and there exist functions $v,w\in L^p(\Omega,\C)$ such that
$$
\int_\Omega v\varphi\,d\mu_\Aa=-\int_\Omega uZ\varphi\,d\mu_\Aa,
\quad
\text{and}
\quad
\int_\Omega w\varphi\,d\mu_\Aa=-\int_\Omega u\overline Z\varphi\,d\mu_\Aa 
$$
for all $\varphi\in C^\infty_0(\Omega,\R)$.
For such a function $u\in HW^{1,p}(\Omega,\C)$, we denote by $Zu$ and $\overline Zu$ the weak horizontal complex derivatives $v$ and $w$. This definition is compatible with the theory of upper gradients on Carnot-Carath\'eodory spaces formulated in \cite{HajK}. A map $f=(f_1,f_I):\Omega\to\Aa$ is said to belong to $HW^{1,p}(\Omega,\Aa)$ if and only if $f_1$, $f_I$ are in $HW^{1,p}(\Omega,\C)$. It is straightforward to define the local horizontal Sobolev spaces $HW^{1,p}_{loc}$.\\
We have that $(\Aa,d_\Aa,\mu_\Aa)$ is a locally $4$-Ahlfors regular space, cf. Proposition 3.5 in \cite{BBP1}.
Combining the last fact with Theorem 11.20 in \cite{HajK}, Theorem 1.1 in \cite{BKR} and Proposition 3.1 in \cite{Shan} we deduce the following result for metric quasiconformal maps on $\Aa$. 
\begin{prop}
    Let $f:\Omega\to\Omega'$ be a quasiconformal mapping between domains $\Omega,\Omega'\subseteq\Aa$. Then the pointwise derivatives $(Re Z)f$ and $(Im Z)f$ exist almost everywhere and coincide with the distributional derivatives almost everywhere.
\end{prop}
In \cite{BBP1}, we proved that the Hausdorff dimension of $(\Aa,d_\Aa)$ is 4. It comes out that the corresponding Sobolev class for quasiconformal mappings in the affine-additive group is $HW^{1,4}_{loc}$.\\
A mapping $f\in HW^{1,4}_{loc}(\Omega,\Aa)$ is called \textit{weakly contact} if the system of p.d.e.s \eqref{contact cond Z, Zbar} holds almost everywhere in $\Omega$.
For such a mapping, we define the \textit{formal tangent map} $$(f_*)_p:T_p\Aa\rightarrow T_{f(p)}\Aa,$$ for almost every $p\in\Omega$. We express it in terms of the bases given by $\mathcal{B}^{CVF}_p=\{Z_p,\overline Z_p,W_p\}$ and $\mathcal{B}^{CVF}_{f(p)}=\{Z_{f(p)},\overline Z_{f(p)},W_{f(p)}\}$ as 
\begin{equation*}
    f_*=\left[\begin{array}{ccc}
         Zf_I/2f_2 & \overline Zf_I/2f_2 & * \\
         Z\overline{f_I}/2f_2 & \overline{Z}\,\overline{f_I}/2f_2 & * \\
         0 & 0 & \sigma
    \end{array}\right],
\end{equation*}
and define the \textit{formal horizontal differential} to be the restriction $D_Hf(p):\mathcal{H}_{p,\Aa}\to\mathcal{H}_{f(p),\Aa}$ given by
\begin{equation*}
    D_Hf(p)=\left[\begin{array}{cc}
         Zf_I/2f_2 & \overline Zf_I/2f_2  \\
         Z\overline{f_I}/2f_2 & \overline{Z}\,\overline{f_I}/2f_2 
    \end{array}\right].
\end{equation*}
Using the commutator relation together with the system \eqref{contact cond Z, Zbar}, we find
$
\sigma=\frac{1}{4f_2^2}\big(|Zf_I|^2-|\overline Z f_I|^2\big)
$ almost everywhere. This yields 
\begin{equation}
  \det (f_*)_p=  \frac{1}{(2f_2(p))^4}\big(|Zf_I(p)|^2-|\overline Z f_I(p)|^2\big)^2 \text{ a.e. in }\Omega.
\end{equation}
Further, set $p\in\Aa$ and $r>0$; we define the volume derivative for $f$ with respect to $\mu_\Aa$ the limit
\begin{equation}
    \mathcal{J}_{\mu_\Aa}(p,f)=\lim_{r\rightarrow0}\frac{\mu_{\Aa}(f(B_\Aa(p,r)))}{\mu_{\Aa}(B_\Aa(p,r))}.
\end{equation}
An interesting property is formulated in the following:
\begin{lem}\label{Jac WC}
Let $f:\Omega\to\Omega'$ be a weakly contact transformation.
Then the identity
$$\mathcal{J}_{\mu_\Aa}(p,f)=\det (f_*)_p,$$ holds almost everywhere in $\Omega$.    
\end{lem}
\begin{proof}
    We observe first that a limit argument and the change of variables theorem induce that at almost every point $p=(a,\lambda+it)\in\Omega$ we have
\begin{equation*}
    \mathcal{J}_{\mu_\Aa}(p,f)=\frac{\lambda^2}{f_2^2(p)}\mathcal{J}(p,f).
\end{equation*}
Here, $\mathcal{J}(p,f)$ corresponds to the determinant of the tangent map $(f_*)_p:T_p\Aa\rightarrow T_{f(p)}\Aa$ considered as a linear map with respect to the canonical bases $\mathcal{B}^{Can}_p=\{\partial_{a|_p},\partial_{\lambda|_p},\partial_{t|_p}\}$ and $\mathcal{B}^{Can}_{f(p)}=\{\partial_{a|_{f(p)}},\partial_{\lambda|_{f(p)}},\partial_{t|_{f(p)}}\}$. The change of bases formula describing the compositions 
$$
\mathcal{B}^{Can}_p\mapsto\mathcal{B}^{CVF}_p\mapsto\mathcal{B}^{CVF}_{f(p)}\mapsto\mathcal{B}^{Can}_{f(p)}
$$ leads to:
\begin{equation*}
    \mathcal{J}_{\mu_\Aa}(p,f)=\frac{1}{(2f_2(p))^4}\big(|Zf_I(p)|^2-|\overline Z f_I(p)|^2\big)^2\,.
\end{equation*}
\end{proof}
We consider the curve family
$$
\Gamma_1=\{\gamma,\gamma:[0,1]\to\Aa\text{ horizontal curve with }\gamma(0)=p \text{ and }|\dot\gamma|_H=1\}
$$
and we define the quantity
$
\|D_H f(p)\|=\max\{|(f\circ\gamma)^\cdot|_H:\gamma\in\Gamma_1\}.
$
Using the complex notation we find the explicit formula
$$
\|D_H f(p)\|=\frac{|Zf_I(p)|+|\overline Z f_I(p)|}{2f_2(p)}\text{ a.e. in }\Omega.
$$
Analytic definition of quasiconformality in $\Aa$ is now in order: 
\begin{defn}[Analytic definition]
    A homeomorphism $f:\Omega\to\Omega'$ between domains $\Omega,\Omega'$ in $\Aa$ is $K$-\textit{quasiconformal} if $f\in HW^{1,4}_{loc}(\Omega,\Aa)$ is weakly contact, and there exists a constant $1\leq K<\infty$ such that
\begin{equation}\label{an qc}
    \|D_H f(p)\|^4\leq K\mathcal{J}_{\mu_\Aa}(p,f) \quad \text{for almost every } p\in\Omega.
\end{equation}
A map is \textit{quasiconformal}, if it is $K$-quasiconformal for some $K$.
\end{defn}
It is straightforward to verify that a $K$-quasiconformal map in the analytic sense has a $K$-quasiconformal inverse. The latter result together with Theorem 3.8 with \cite{KW} induce that a homeomorphism is quasiconformal according the analytic sense if and only if it is quasiconformal according to the metric one.\\
It can be proven that $\mathcal{J}_{\mu_\Aa}(\cdot,f)\neq0$ a.e. for a quasiconformal mapping $f$. The above considerations show that for a quasiconformal map $f:\Omega\to\Omega'$ between domains in the affine-additive group the following holds:
$$
K(p,f)^2=\frac{\|D_H f(p)\|^4}{\mathcal{J}_{\mu_\Aa}(p,f)}=\left(\frac{|Zf_I(p)|+|\overline Z f_I(p)|}{|Zf_I(p)|-|\overline Z f_I(p)|}\right)^2\text{ a.e. in }\Omega.
$$
By setting $K(\cdot,f)^2=1$ at the points where $\mathcal{J}_{\mu_\Aa}(\cdot,f)=0$, we obtain that $K(\cdot,f)^2$ is a measurable function on $\Omega$ which is finite almost everywhere.
A quasiconformal map $f:\Omega\to\Omega'$ between domains in the affine-additive group is called \textit{orientation preserving} if 
$$
\det D_H f(p)>0 \quad\text{for almost every }p\in\Omega.
$$
By recalling the expressions defined in the introduction
$$\mu_f(p)=\frac{\overline{Z}f_I(p)}{Zf_I(p)},\quad K(p,f)=\frac{|Zf_I(p)|+|\overline Z f_I(p)|}{|Zf_I(p)|-|\overline Z f_I(p)|}$$
and by defining 
$$
\|\mu_f\|_\infty=\esssup_p|\mu_f(p)|,\quad K_f=\esssup_p K(p,f),
$$
the explicit relation between $\|\mu_f\|_\infty$ and $K_f$ can now be written explicitly as follows:
\begin{equation}\label{K_f}
    K(p,f)=\frac{1+|\mu_f(p)|}{1-|\mu_f(p)|},\quad K_f=\frac{1+\|\mu_f\|_\infty}{1-\|\mu_f\|_\infty}.
\end{equation}
We state below a change of variable formula for integration in the case of quasiconformal mappings on the affine-additive group whose proof can be found in \cite{Bubani}.
\begin{prop}
    Let $f:\Omega\to\Omega'$ be a quasiconformal mapping between domains $\Omega,\Omega'\subseteq\Aa$. Then the following transformation formula holds: if $u:\Aa\to\R$ is a measurable non-negative function, then the function $p\mapsto(u\circ f)(p)J_{\mu_\Aa}(p,f)$ is measurable and we have
    $$
    \int_\Omega(u\circ f)(p)\mathcal{J}_{\mu_\Aa}(p,f)\,d\mu_\Aa(p) =\int_{\Omega'} u(q)\,d\mu_\Aa(q).
    $$
\end{prop}
\subsection{Modulus of curve families.}
\subsubsection{Curves in the affine-additive group.}
Any curve $\gamma$ in $\Aa$ shall be always
considered continuous. The points on a curve $\gamma : [c, d] \to \Aa$ are
denoted by 
$$
\gamma(s)=(\gamma_1(s),\gamma_I(s))\in\R\times\mathbf{H}^1_\C.
$$
An absolutely continuous curve $\gamma:[c,d]\to\Aa$ (in the Euclidean sense)
is called \textit{horizontal} if
$$
\dot\gamma(s)\in\ker\vartheta_{\gamma(s)}\quad \text{for almost every } s\in[c,d].
$$
The length of a horizontal curve corresponds to
$$
\ell(\gamma)=\int_c^d|\dot\gamma(s)|_H\,ds.
$$
We say that $\gamma$ is \textit{rectifiable} when $\ell(\gamma)$ is finite, moreover we say that $\gamma$ is \textit{locally rectifiable} when all its closed sub-curves are rectifiable.\\
If $\gamma:[c,d]\to\Aa$ is a rectifiable curve, the line integral over $\gamma$ of a Borel function $\rho:\Aa\to[0,\infty]$ is defined as
$$
\int_\gamma \rho \,d\ell=\int_c^d \rho(\gamma(s))|\dot\gamma(s)|_H\,ds,
$$
and in case $\gamma$ is only locally rectifiable, we set
$$
\int_\gamma \rho \,d\ell=\sup\left\{\int_{\gamma'} \rho\,d\ell:\gamma'\text{ is a rectifiable subcurve of }\gamma\right\}.
$$

\subsubsection{Modulus of a curve family}
The definition for the conformally invariant
$4$-modulus of a family $\Gamma$ of curves in $\Aa$ has been given in the introduction, see \eqref{def 4-Mod}. For curves $\gamma : (a, b) \to \Aa$
we shall employ the notion of local rectifiability.\\
It is worth to say that a family which consists only of curves that are not locally rectifiable has modulus zero, \cite{HKST-book}. All quasiconformal
mappings of the affine-additive group are absolutely continuous on almost every curve, \cite{HKST-paper}, \cite{Shan}. To sum up, given a quasiconformal map $f : \Omega \to \Omega'$ between domains
in the affine-additive group and given a family $\Gamma$ of closed rectifiable curves in $\Omega$, we have
$$
\Mod_4\left(\gamma\in\Gamma:f\circ\gamma\text{ not absolutely continuous}\right)=0.
$$

\bibliographystyle{plain}  
\bibliography{My_Library.bib}    

\Addresses

\end{document}